\documentclass[12pt]{amsart}
\usepackage[all]{xypic}
\usepackage{tikz}
\usetikzlibrary{arrows}
\usetikzlibrary{decorations.markings}

\usepackage{graphicx}
\usepackage{bm}
\usepackage{epsf}
\usepackage{verbatim} 
\usepackage{amsmath}
\usepackage{amsfonts}
\usepackage{amssymb}
\usepackage{mathrsfs}
\usepackage{amsthm}
\usepackage{newlfont}

 \newtheorem{thm}{Theorem}

 \newtheorem{conj}[thm]{Conjecture}
 
 \newtheorem{lem}[thm]{Lemma}
 
 \newtheorem{prop}[thm]{Proposition}
 
 \theoremstyle{definition}
 \newtheorem{defn}[thm]{Definition}
 
 \theoremstyle{definition}
 \newtheorem{notn}[thm]{Notation}

 \theoremstyle{remark}
 \newtheorem{rem}[thm]{Remark}
 
 \theoremstyle{definition}
 \newtheorem{example}[thm]{Example}

\numberwithin{thm}{section}
\numberwithin{equation}{section}

 \newcommand{\Div}{\mathrm{Div}}
 \newcommand{\End}{\mathrm{End}}

 \newcommand{\ord}{\mathrm{ord}}

 \newcommand{\GL}{\mathrm{GL}}
 \newcommand{\PGL}{\mathrm{PGL}}

 \newcommand{\tor}{\mathrm{tor}}

 \renewcommand{\mod}{\mathrm{mod}}
 
 \newcommand{\fb}{\mathfrak b}
 
 \newcommand{\fp}{\mathfrak p}
 \newcommand{\fq}{\mathfrak q}
 \newcommand{\fr}{\mathfrak r}
 \newcommand{\fn}{\mathfrak n}
 \newcommand{\fm}{\mathfrak m}
 \newcommand{\fd}{\mathfrak d}

 \newcommand{\fE}{\mathfrak E}

 \newcommand{\cO}{\mathcal{O}}

 \renewcommand{\cH}{\mathcal{H}}

 \newcommand{\cE}{\mathcal{E}}
 
 \newcommand{\cC}{\mathcal{C}}

 \newcommand{\cI}{\mathcal{I}}
 \newcommand{\cT}{\mathcal{T}}

\newcommand{\sT}{\mathscr{T}}

\newcommand{\bone}{{\bm{1}}}
\newcommand{\beps}{{\bm{\epsilon}}}

 \newcommand{\C}{\mathbb{C}}
 \newcommand{\F}{\mathbb{F}}
 \newcommand{\Q}{\mathbb{Q}}
 \newcommand{\W}{\mathbb{W}}
 \newcommand{\Z}{\mathbb{Z}}
 \newcommand{\E}{\mathbb{E}}

 \newcommand{\p}{\mathbb{P}}
 \newcommand{\T}{\mathbb{T}}
 
 \newcommand{\N}{\mathbb{N}}

 \newcommand{\tE}{\widetilde{E}}
 \newcommand{\ctH}{\widetilde{\mathcal{H}}}
 \newcommand{\ctE}{\widetilde{\mathcal{E}}}

 \newcommand{\eps}{\varepsilon}
 \newcommand{\G}{\Gamma}
 
 \newcommand{\bs}{\setminus}
 \newcommand{\Fi}{F_\infty}

 \newcommand{\la}{\lambda}

\begin{document}

\title[Rational torsion subgroups and Eisenstein cochains]
{The rational torsion subgroups of Drinfeld modular Jacobians and Eisenstein pseudo-harmonic cochains}

\author{Mihran Papikian}
\address{Department of Mathematics, Pennsylvania State University, University Park, PA 16802, U.S.A.}
\email{papikian@psu.edu}
\author{Fu-Tsun Wei}
\address{Institute of Mathematics, Academia Sinica, 6F, Astronomy-Mathematics Building, No. 1, Sec. 4, Roosevelt Road, Taipei 10617, Taiwan}
\email{ftwei@math.sinica.edu.tw}

\thanks{The first author's research was partially supported by grants from the Simons Foundation (245676) and the National Security Agency 
(H98230-15-1-0008).} 




\begin{abstract} Let $\fn$ be a square-free ideal of $\F_q[T]$. We study the 
rational torsion subgroup of the Jacobian variety $J_0(\fn)$ of the Drinfeld modular curve $X_0(\fn)$. 
We prove that for any prime number $\ell$ not dividing $q(q-1)$, the $\ell$-primary part of this group 
coincides with that of the cuspidal divisor class group. We further 
determine the structure of the $\ell$-primary part of the cuspidal divisor class group for any prime $\ell$ not dividing $q-1$. 
\end{abstract}

\maketitle


\section{Introduction} 

\subsection{Rational torsion of classical modular Jacobians} Let $N\geq 1$ be a positive integer. Let $X_0(N)$ be the 
modular curve over $\Q$ parametrizing the isomorphism classes of (generalized) elliptic curves with $\G_0(N)$-structures. 
The rational torsion subgroup $\cT(N):=J_0(N)(\Q)_\tor$ 
of the Jacobian variety $J_0(N)$ of $X_0(N)$ is a finite group by the Mordell-Weil theorem. The cuspidal 
divisor class group $\cC(N)$ of $J_0(N)$, i.e., the group generated by the classes of differences of two cusps of $X_0(N)$, 
is also finite by a theorem of Manin and Drinfeld.  In the early 1970s, for $N=p$ 
prime, Ogg computed that $\cC(p)\cong \Z\big/\frac{p-1}{(p-1,12)}\Z$ and conjectured that $\cT(p)=\cC(p)$; see \cite{OggBAMS}. 
In his seminal paper \cite{Mazur}, Mazur proved this conjecture by a detailed study of the 
Eisenstein ideal of the Hecke ring of level $p$. 

When $N$ is square-free, all cusps of $X_0(N)$ are rational over $\Q$. 
Therefore, $\cC(N)\subseteq \cT(N)$. A natural generalization of Ogg's conjecture then would predict an equality $\cC(N)=\cT(N)$.  
Recently, Ohta \cite{Ohta} proved the following theorem toward this conjecture (given a finite abelian group $G$, we denote by $G_\ell$ 
the maximal $\ell$-primary subgroup of $G$): 
\begin{thm}\label{thmMOhta}
Let $N$ be a square-free positive integer. We have $\cC(N)_\ell=\cT(N)_\ell$ for all 
prime numbers $\ell\geq 3$ when $N$ is not divisible by $3$; and for all prime numbers $\ell\geq 5$ 
when $N$ is divisible by $3$. 
\end{thm}

The proof of this theorem again relies on the study of Eisenstein ideals. Ohta 
computes the index of the Eisentein ideal in an appropriate Hecke algebra, up to a power of $2$, which turns out to be 
the order of $\cC(N)$ by a formula of Takagi \cite{Takagi}, again up to a power of $2$. One 
important observation that Ohta makes is that the Hecke algebra best suited for the purpose of proving $\cT(N)=\cC(N)$ 
in the square-free case is the algebra generated by the Hecke operators $T_p$, for prime numbers $p$ 
not dividing $N$, and the Atkin-Lehner involutions $W_p$ for $p\mid N$ (instead of $U_p$ operators). 
The restriction $3\nmid N$ in Theorem \ref{thmMOhta} in the case $\ell=3$ is of technical nature, partly arising 
from the existence of constant modular forms over $\F_3$ of level $1$ and weight $2$. 
On the other hand, the proof of $\cC(N)_2=\cT(N)_2$ is beyond the reach of the method used by Ohta. 
Even in the prime level case considered by Mazur \cite{Mazur}, proving $\cC(p)_2=\cT(p)_2$ requires deeper techniques 
related to the ring-theoretic properties of the Hecke algebra.
 
\subsection{Rational torsion of Drinfeld modular Jacobians}
Let $\F_q$ be a finite field with $q$ 
elements, where $q$ is a power of a prime number $p$. Let $A=\F_q[T]$ 
be the ring of polynomials in indeterminate $T$ with coefficients 
in $\F_q$, and $F=\F_q(T)$ the field of fractions of $A$. 
The degree map $\deg: F\to \Z\cup \{-\infty\}$, which associates 
to a non-zero polynomial its degree in $T$ and $\deg(0)=-\infty$, defines 
a norm on $F$ by $|a|:=q^{\deg(a)}$; the corresponding place of $F$ 
is usually called the \textit{place at infinity} and is denoted by $\infty$. 
We also define a norm and degree on the ideals of $A$ by $|\fn|:=\#(A/\fn)$ and $\deg(\fn):=\log_q|\fn|$. 
Let $\Fi$ denote the completion of $F$ at $\infty$, and $\C_\infty$ denote the 
completion of an algebraic closure of $\Fi$. The \textit{Drinfeld half-plane} $\Omega:=\C_\infty - \Fi$   
has a natural structure of a smooth connected rigid-analytic space over $\Fi$; see \cite[$\S$1]{GR}. 

Let $\fn\lhd A$ be a non-zero ideal. The level-$\fn$ \textit{Hecke congruence subgroup} of $\GL_2(A)$  
$$
\G_0(\fn):=\left\{\begin{pmatrix} a & b \\  c & d\end{pmatrix}\in \GL_2(A)\ \bigg|\ c\in \fn \right\}  
$$
plays a central role in this paper. This group acts on $\Omega$ via linear fractional transformations. 
Drinfeld proved in \cite{Drinfeld} that the quotient $\G_0(\fn)\bs \Omega$ 
is the space of $\C_\infty$-points of an affine curve $Y_0(\fn)$ defined over $F$,  
which is a moduli space of rank-$2$ Drinfeld modules.  
The unique smooth projective curve over $F$ containing $Y_0(\fn)$ as an 
open subvariety is denoted by $X_0(\fn)$. The \textit{cusps} of $X_0(\fn)$ are the 
finitely many geometric points of the complement of $Y_0(\fn)$ in $X_0(\fn)$. 
Let $J_0(\fn)$ be the Jacobian variety of $X_0(\fn)$. 
The \textit{cuspidal divisor group} $\cC(\fn)$ 
is the subgroup of $J_0(\fn)$ generated by the classes of divisors $c-c'$, where $c, c'$  
run through the set of cusps of $X_0(\fn)$.  
It is known that $\cC(\fn)$ is a finite group; see \cite{GekelerIJM}. 
By the Lang-N\'eron theorem, the group of $F$-rational points of $J_0(\fn)$ is 
finitely generated, in particular, its torsion subgroup $\cT(\fn):=J_0(\fn)(F)_\tor$ is finite. 
Finally, it is known that, when $\fn$ is square-free, all cusps of $X_0(\fn)$ are rational over $F$; see \cite[Prop. 6.7]{Invariants}. 
Therefore, $\cC(\fn)\subseteq \cT(\fn)$. In analogy with generalized Ogg conjecture that we discussed earlier, one can propose the following: 

\begin{conj}\label{conjGOC}
For a square-free ideal $\fn\lhd A$, we have $\cC(\fn)= \cT(\fn)$.
\end{conj}

This statement for $\fn=\fp$ prime was proved by P\'al \cite{Pal}, following Mazur's method. 

Now let $\fn=\prod_{i=1}^s\fp_i$ be a square-free ideal in $A$ with the given prime decomposition. 
Let 
$$
\E=\{(\eps_1, \dots, \eps_s)\ |\ \eps_i=\pm 1,\ 1\leq i\leq s\}.
$$ 
We single out two (not necessarily distinct) elements of $\E$:
\begin{align*}
\beps_{H} &:=((-1)^{\deg(\fp_1)}, \dots, (-1)^{\deg(\fp_s)}), \text{ and} \\ 
\bone &:=(1,1,\dots, 1).
\end{align*}
For $\beps=(\eps_1, \dots, \eps_s)\in \E$, define 
$$
N(\beps)= 
\begin{cases}
1 & \text{if $\beps=\bone$};\\
\prod_{i=1}^s(1+\eps_i|\fp_i|) & \text{if $\beps=\beps_{H}$ and $\beps\neq \bone$}; \\
\frac{1}{q+1}\prod_{i=1}^s(1+\eps_i|\fp_i|) & \text{if $\beps\neq\beps_{H}$ and $\beps\neq \bone$}.  
\end{cases}
$$
The main result of this paper is the following: 
\begin{thm}\label{thmMainT} Let $\fn=\prod_{i=1}^s \fp_i$ be a square-free ideal in $A$. 
Let $\ell$ be a prime number not dividing $q(q-1)$. Then 
$$
\cT(\fn)_\ell=\cC(\fn)_\ell \cong \bigoplus_{\beps\in \E} \Z_\ell/N(\beps)\Z_\ell. 
$$ 
\end{thm}

The method that we use falls short of proving $\cT(\fn)_\ell=\cC(\fn)_\ell$ for primes $\ell$ dividing $(q-1)$; these primes 
are the analogues of $\ell=2$ over $\Q$.  The prime $\ell=p$, being the characteristic of $F$, 
has peculiar properties, as far as Conjecture \ref{conjGOC} is concerned. On the one hand, it is not hard to 
prove the following:
\begin{prop}
Assume $\fn\lhd A$ is square-free and there is a prime divisor $\fp\mid \fn$ such that $\deg(\fn/\fp)\leq 2$. Then $\cT(\fn)_p=0$. 
\end{prop}
\begin{proof}
If $\deg(\fn/\fp)\leq 2$, then $J_0(\fn)$ has purely toric reduction at $\fp$. 
Moreover, the component group $\Phi_\fp$ of the N\'eron model of $J_0(\fn)$ at $\fp$ 
has order coprime to $p$; see \cite[Thm. 5.3]{PW1}. On the other hand, 
because the reduction is purely toric, $\cT(\fn)_p$ injects into $\Phi_\fp$; cf. \cite[Lem. 7.13]{Pal}.
Hence $\cT(\fn)_p=0$. 
\end{proof}

On the other hand, when there is no prime dividing $\fn$ at which $J_0(\fn)$ has purely toric reduction, 
it is not clear to us how to analyse $\cT(\fn)_p$. More precisely, it not clear whether the theory of the 
Eisenstein ideal is applicable at all to the study of $\cT(\fn)_p$. (When $\ell\neq p$, the Eichler-Shimura 
congruence relation implies that $T_\fp-(|\fp|+1)$, $\fp\nmid \fn$, annihilates $\cT(\fn)_\ell$, 
hence one can use the Eisenstein ideal to study $\cT(\fn)_\ell$.) In any case, we prove the following (see Proposition \ref{propCp=0}): 

\begin{prop}
If $\fn$ is square-free, then $\cC(\fn)_p=0$. Therefore, if Conjecture \ref{conjGOC} is valid, then $\cT(\fn)_p=0$. 
\end{prop}

\begin{rem} We proved in \cite{PW2} that when $\fn$ is not square-free, and is not equal to a square of a prime, $\cT(\fn)_p\neq 0$. 
Note that in \cite{PW2} we constructed explicit elements in $\cT(\fn)_p$ using cuspidal divisors, which 
can be shown to be annihilated by the Eisenstein ideal. 
\end{rem}

Now we give an outline of the proof of Theorem \ref{thmMainT}. The Atkin-Lehner 
involutions form a group $\W\cong (\Z/2\Z)^s$ which acts on $\cT(\fn)$ and $\cC(\fn)$. 
Away from the $2$-primary components, one can decompose $\cT(\fn)$ and $\cC(\fn)$ into direct sums of $\W$-eigenspaces, 
each eigenspace corresponding to some $\beps\in \E$. Hence, it is enough to show that 
$\cT(\fn)_\ell^\beps=\cC(\fn)_\ell^\beps\cong \Z_\ell/N(\beps)\Z_\ell$ for all $\ell\nmid q(q-1)$ and $\beps=(\eps_1, \dots, \eps_s)\in \E$, 
where $\cT(\fn)_\ell^\beps$ denotes the largest direct summand of $\cT(\fn)_\ell$ on which $W_{\fp_i}$ 
acts by $\eps_i$, $1\leq i\leq s$, and similarly for $\cC(\fn)_\ell^\beps$. 
Since $J_0(\fn)$ has split toric reduction at $\infty$, the $\ell$-primary subgroup $\cT(\fn)_\ell$, $\ell\nmid q(q-1)$, 
maps injectively into the component group $\Phi_\infty$ of the N\'eron model of $J_0(\fn)$ at $\infty$. 
Then, using the rigid-analytic uniformization of $J_0(\fn)$ at $\infty$ and the Eichler-Shimura relations, one 
shows that the image of $\cT(\fn)_\ell^\beps$ in $\Phi_\infty$ can be identified with a subspace of $\cE_0(\fn, \Z/\ell^n\Z)^\beps$ for any sufficiently large $n$. 
Here  $\cE_0(\fn, \Z/\ell^n\Z)^\beps$ denotes the module of $\G_0(\fn)$-invariant $\Z/\ell^n\Z$-valued 
cuspidal harmonic cochains on which $T_\fp$ ($\fp\nmid \fn$) acts by multiplication by $|\fp|+1$, and $\W$ acts by $\beps$. 
We study the space $\cE_0(\fn, \Z/\ell^n\Z)^\beps$ in Section \ref{sEPHC}, where we show that it is 
generated by the reduction modulo $\ell^n$ of certain Eisenstein series with constant Fourier coefficient $qN(\beps)$. 
We then use this to identify $\cT(\fn)_\ell^\beps$ with a subgroup of $\Z_\ell/N(\beps)\Z_\ell$. 
On the other hand, in Section \ref{sCDG}, we construct an explicit element in $\cC(\fn)_\ell^\beps$, and 
use it to show that $\cC(\fn)_\ell^\beps$ contains a subgroup isomorphic to $\Z_\ell/N(\beps)\Z_\ell$. 
Now Theorem \ref{thmMainT} immediately follows by comparing the orders of $\cC(\fn)_\ell^\beps$ and $\cT(\fn)_\ell^\beps$. 
We note that the crucial idea of using the Atkin-Lehner involutions instead of $U_\fp$ operators to 
analyse the Eisenstein harmonic cochains was inspired by Ohta's paper \cite{Ohta}. Also, the trick 
with mapping $\cT(\fn)_\ell$ into $\Phi_\infty$ was first used by P\'al in \cite{Pal} in the prime level case.  

\begin{rem}
In \cite{PW1}, we proved a result toward Conjecture \ref{conjGOC} when $\fn$ is a product 
of two distinct primes. In that paper, which was mostly written in 2012, we used $U_\fp$ 
operators instead of Atkin-Lehner involutions to analyse $\cE_0(\fn, \Z/\ell^n\Z)$. That approach 
is technically more complicated and leads to a weaker result than Theorem \ref{thmMainT} specialized to $s=2$.  
On the other hand, in \cite{PW1}, we determined the structure of $\cC(\fp_1\fp_2)_\ell$ for all $\ell$, 
including $\ell\mid (q-1)$. 
\end{rem}

The Eisenstein ideal does not explicitly appear in this paper, although it is in the background of the analysis of  
$\cE_0(\fn, \Z/\ell^n\Z)^\beps$. 
The most technical part of the paper is the analysis of $\cE_0(\fn, \Z/\ell^n\Z)^{\beps_H}$ for $\ell \mid (q+1)$, 
which occupies a large portion of Section \ref{sEPHC}. 
The odd primes dividing $q+1$ are somewhat similar to $\ell=3$ in Ohta's setting; for example, there 
is a non-trivial $\GL_2(A)$-invariant $\Z/(q+1)\Z$-valued harmonic cochain, 
although there are no such cochains with values in rings where $q+1$ is invertible. 

For our purposes we found it convenient 
to generalize the notion of harmonic cochain. In Section \ref{sAL}, we introduce what we call \textit{pseudo-harmonic cochains}, 
which are functions on the edges of the Bruhat-Tits tree of $\PGL_2(\Fi)$ satisfying the flow condition 
of harmonic cochains but which are not necessarily alternating. It turns out that there is a $\GL_2(A)$-invariant $\Z$-valued 
pseudo-harmonic cochain $\tE$, which in our setting plays a role of the classical Eisenstein series $E_2(z)=\sum_{c,d\in \Z}'(cz+d)^{-2}$.
(Another function field analogue of $E_2$ was introduced by Gekeler in \cite{Improper} under the name of \textit{improper 
Eisenstein series}; we explain the relationship between our $\tE$ and Gekeler's Eisenstein series in Remark \ref{remH}.) 

\subsection*{Acknowledgements} This work was carried out while the first 
 author was visiting the Taida Institute for Mathematical Sciences in Taipei. 
 He thanks Professor Jing Yu for the invitation. He also thanks the institute 
for its hospitality and good working conditions.


\section{Pseudo-harmonic cochains and Hecke operators}\label{sAL} 

\subsection{Notation} 
Besides $\infty$, the other places of $F$ are in bijection with the non-zero prime ideals of $A$.  
Given a place $v$ of $F$, we denote by $F_v$ the completion of $F$ at $v$, by $\cO_v$ 
the ring of integers of $F_v$, and by $\F_v$ the residue field of $\cO_v$. 
We fix $\pi_\infty:=T^{-1}$ as a uniformizer of $\cO_\infty$. 

Let $R$ be a commutative ring with identity. We denote by $R^\times$ the group of 
multiplicative units of $R$. Let 
$\GL_n(R)$ be the group of $n\times n$ matrices over $R$ whose determinant is in $R^\times$, and $Z(R)\cong R^\times$ 
the subgroup of  $\GL_n(R)$ consisting of scalar matrices. 


Given an abelian group $H$ and an integer $n$, $H[n]$ is the kernel of multiplication by 
$n$ in $G$. For a prime number $\ell$, $H_\ell$ denotes the $\ell$-primary component of $H$. 

Given an ideal $\fn\lhd A$, by abuse of notation, we denote by the same symbol the unique 
monic polynomial in $A$ generating $\fn$. It will always be clear from the context 
in which capacity $\fn$ is used; for example, if $\fn$ appears in a matrix, column vector, or a 
polynomial equation, then the monic polynomial is implied. 
The prime ideals $\fp\lhd A$ are always assumed to be non-zero. Given two ideals $\fn, \fm$ of $A$,  $(\fn, \fm)$ 
stands for the greatest common divisor of $\fn$ and $\fm$, and  
$\fm\parallel \fn$ means that $\fm$ divides $\fn$ and $(\fm, \fn/\fm)=1$. 

\subsection{Pseudo-harmonic cochains} \label{sec2.1}

Let $G$ be an oriented connected graph in the sense of Definition 1 of $\S$2.1 in \cite{SerreT}. 
We denote by $V(G)$ and $E(G)$ its set of vertices and edges, respectively. 
For an edge $e\in E(G)$, let $o(e)$, $t(e)\in V(G)$ and $\bar{e}\in E(G)$ be its 
origin, terminus and inversely oriented edge, respectively. 

\begin{defn}\label{defnHarmG}
Let $R$ be a commutative ring with identity. An $R$-valued \textit{pseudo-harmonic cochain} on $G$ 
is a function $f: E(G)\to R$ that satisfies 
\begin{equation}\label{eq-pharm}
\sum_{\substack{e'\in E(G)\\ t(e')=o(e),\ e'\neq \bar{e}}} f(e')=f(e)\quad \text{for all $e\in E(G)$}.
\end{equation}
A pseudo-harmonic cochain is called \textit{harmonic} if it is alternating:  
\begin{equation}\label{eq-alt}
f(e)+f(\bar{e})=0\quad \text{for all $e\in E(G)$}.
\end{equation}
Note that (\ref{eq-alt}) makes (\ref{eq-pharm}) equivalent to the following 
$$
\sum_{\substack{e\in E(G)\\ t(e)=v}} f(e)=0\quad \text{for all $v\in V(G)$}.
$$
Denote by $\ctH(G, R)$ (resp. $\cH(G, R)$) the $R$-module of pseudo-harmonic (resp. harmonic) cochains on $G$. 
\end{defn}

\begin{lem}\label{lem 2.1FT} Assume $G$ is connected. Given $f\in \ctH(G, R)$, there is a constant $c\in R$ such that 
$$f(e) + f(\bar{e}) = c, \quad \forall e \in E(G).$$
\end{lem}
\begin{proof}
Since $G$ is connected, it suffices to show that for every $e_1,e_2 \in E(G)$ with $t(e_1) = o(e_2)$ and $e_1 \neq \bar{e}_2$,
$$f(e_1)+f(\bar{e}_1) = f(e_2) + f(\bar{e}_2).$$
By definition 
$$f(e_2) = f(e_1) + \sum_{\substack{e' \in E(G) \\ t(e') = o(e_2),\ e' \neq \bar{e}_2, e_1}} f(e')$$
and
$$
f(\bar{e}_1) = f(\bar{e}_2) + \sum_{\substack{e' \in E(G)\\ t(e') = o(\bar{e}_1),\ e' \neq e_1, \bar{e}_2}} f(e')
= f(\bar{e}_2)+ \sum_{\substack{e' \in E(G)\\ t(e') = o(e_2),\ e' \neq  \bar{e}_2, e_1}} f(e').
$$
Therefore, $f(e_2) - f(e_1) = f(\bar{e}_1) - f(\bar{e}_2)$, and the result follows. 
\end{proof}

\begin{rem}
The pseudo-harmonic cochains on the Bruhat-Tits tree $\sT$ are a special case of metaplectic forms over function fields 
whose general theory is developed in \cite{WeiMA}. In the terminology of \cite{WeiMA},  pseudo-harmonic cochains 
are the ``weight-2'' forms. 
\end{rem}

The graphs that we consider in this paper are the Bruhat-Tits tree $\sT$ of $\PGL_2(\Fi)$ and the 
quotients of $\sT$. We recall the definition and introduce some notation for later use. 
Fix a uniformizer $\pi_\infty$ of $\Fi$. 
The sets of vertices $V(\sT)$ and edges $E(\sT)$ are the cosets $\GL_2(\Fi)/Z(\Fi)\GL_2(\cO_\infty)$ 
and $\GL_2(\Fi)/Z(\Fi)\cI_\infty$, respectively, where $\cI_\infty$ is the Iwahori group:
$$
\cI_\infty=\left\{\begin{pmatrix} a & b\\ c & d\end{pmatrix}\in \GL_2(\cO_\infty)\ \bigg|\ c\in \pi_\infty\cO_\infty\right\}. 
$$
The matrix $\begin{pmatrix} 0 & 1\\ \pi_\infty & 0\end{pmatrix}$ 
normalizes $\cI_\infty$, so the multiplication from the right by this matrix on $\GL_2(\Fi)$ 
induces an involution on $E(\sT)$; this involution is $e\mapsto \bar{e}$. 
The matrices 
\begin{equation}\label{eq-setM}
E(\sT)^+=\left\{\begin{pmatrix} \pi_\infty^k & u \\ 0 & 1\end{pmatrix}\ \bigg|\
\begin{matrix} k\in \Z\\ u\in \Fi,\ u\ \mod\ \pi_\infty^k\cO_\infty\end{matrix}\right\}
\end{equation}
are in distinct left cosets of $\cI_\infty Z(\Fi)$, and there is a disjoint decomposition  (cf.\ \cite[(1.6)]{Improper})
\begin{equation}\label{eq-edge+}
E(\sT)=E(\sT)^+\bigsqcup E(\sT)^+\begin{pmatrix} 0 & 1\\ \pi_\infty & 0\end{pmatrix}. 
\end{equation}
We call the edges in $E(\sT)^+$ \textit{positively oriented}.

The group $\GL_2(\Fi)$ naturally acts on $E(\sT)$ by left multiplication. 
This induces an action on the group of $R$-valued functions on $E(\sT)$: 
for a function $f$ on $E(\sT)$ and $\gamma\in \GL_2(\Fi)$ we define the function $f|\gamma$ on $E(\sT)$ by 
$(f|\gamma)(e)=f(\gamma e)$. 
It is clear from the definition that $f|\gamma$ is pseudo-harmonic (resp. harmonic) if $f$ is pseudo-harmonic (resp. harmonic). 

\begin{defn} Let $\G$ be a subgroup of $\GL_2(\Fi)$. 
Denote by $\ctH(\sT, R)^\G$ (resp. $\cH(\sT, R)^\G$) the $R$-submodule of $\G$-invariant pseudo-harmonic (resp. harmonic) cochains, 
i.e., $f|\gamma=f$ for all $\gamma\in \G$.  
The module of $R$-valued \textit{cuspidal} 
harmonic cochains for $\G$, denoted $\cH_0(\sT, R)^\G$, is the 
submodule of $\cH(\sT, R)^\G$ consisting of functions 
which have compact support modulo $\G$. 
Let $\cH_{00}(\sT, R)^\G$ 
denote the image of $\cH_0(\sT, \Z)^\G\otimes R$ in $\cH_0(\sT, R)^\G$. To simplify the notation, 
we denote the $R$-module of pseudo-harmonic (resp. harmonic, cuspidal) $\G_0(\fn)$-invariant cochains by 
$$
\ctH(\fn, R)\supset \cH(\fn, R)\supset \cH_0(\fn, R)\supset \cH_{00}(\fn, R). 
$$ 
In general, all these inclusions are strict; in particular $\cH_0(\fn, R)\neq \cH_{00}(\fn, R)$, although 
one can show that $\cH_0(\fn, R)= \cH_{00}(\fn, R)$ if $R$ is flat over $\Z$ or $q(q^2-1)\in R^\times$; see \cite[$\S$2.1.]{PW1}.
(We could have also defined cuspidal pseudo-harmonic $\G_0(\fn)$-invariant cochains, but such cochains are necessarily harmonic 
by Lemma \ref{lem 2.1FT}, i.e., $\ctH_0(\fn, R)=\cH_0(\fn, R)$.)
\end{defn}

 It is known that the quotient graph $\G_0(\fn)\bs \sT$ is the edge disjoint union 
$$
\G_0(\fn)\bs \sT = (\G_0(\fn)\bs \sT)^0\cup \bigcup_{s\in \G_0(\fn)\bs \p^1(F)} h_s
$$
of a finite graph $(\G_0(\fn)\bs \sT)^0$ with a finite number of half-lines $h_s$, called \textit{cusps}; 
cf. Theorem 2 on page 106 of \cite{SerreT}. (A half-line is a graph as in Figure \ref{Fig1}.)
The cusps are in bijection with the orbits of the natural action of $\G_0(\fn)$ on $\p^1(F)$; cf. Remark 2 
on page 110 of \cite{SerreT}. It is clear that $f\in \cH(\fn, R)$ is cuspidal if and only if it eventually vanishes on each $h_s$. 

\begin{example}\label{example1} 
\begin{figure}
\begin{tikzpicture}[->, >=stealth', semithick, node distance=1.5cm, inner sep=.5mm, vertex/.style={circle, fill=black}]

\node[vertex] (0) [label=below:$v_0$]{};
  \node[vertex] (1) [right of=0, label=below:$v_1$] {}; 
  \node[vertex] (2) [right of=1, label=below:$v_2$] {};
  \node[vertex] (3) [right of=2, label=below:$v_3$] {};
  \node[] (4) [right of=3] {};

\path[]
    (0) edge  (1) (1) edge (2) (2) edge (3) (3) edge[dashed] (4);   
\end{tikzpicture}
\caption{$\GL_2(A)\bs \sT$}\label{Fig1}
\end{figure}
The quotient graph $\GL_2(A)\bs \sT$ is a half-line depicted in Figure \ref{Fig1}, 
where the vertex $v_i$ ($i\geq 0$) 
is the image of $\begin{pmatrix} \pi_\infty^{-i} & 0\\ 0 & 1\end{pmatrix}\in V(\sT)$.  
Denote the edge with origin $v_i$ and terminus $v_{i+1}$ by $e_i$. It is clear that $f\in \ctH(1, R)$ defines a 
function on $\GL_2(A)\bs \sT$, and $f$ itself can be uniquely recover from that function. Hence we consider $f\in \ctH(1, R)$ 
as a function on the quotient graph. The stabilizers 
of vertices and edges of $\GL_2(A)\bs \sT$ are well-known, cf. \cite[p. 691]{GN}. 
Using this one easily computes that (\ref{eq-pharm}) is equivalent to the following relations (cf. \cite[Example 2.4]{PW1}): 
\begin{align*}
 q\cdot f(\bar{e}_0)&=f(e_0), \\ q\cdot f(e_{i+1})&=f(e_i), \qquad \forall i\geq 0,
 \\ f(\bar{e}_{i+1})+(q-1)f(e_i)&=f(\bar{e}_i), \qquad \forall i\geq 0. 
\end{align*}
Therefore, if we put $f(\bar{e}_0)=\alpha$, then $f\in \ctH(1, R)$ if and only if for all $i\geq 0$
\begin{align*}
f(e_i) &=q^{i+1}\alpha,\\
f(\bar{e}_i) &=(1+q-q^{i+1})\alpha. 
\end{align*}
In particular, $f(e_i)+f(\bar{e}_i) = (q+1)\alpha$ for all $i\geq 0$. We conclude that $\ctH(1, R)\cong R$ 
is spanned by the function $\tE$ with $\tE(\bar{e}_0)=1$, $\cH(1, R)\cong R[q+1]$,   
and $\cH_0(1,R)=\cH_{00}(1,R)=0$. 
\end{example} 

\begin{rem}
When $(q+1)$ is invertible in $R$, it is easy to see from Lemma \ref{lem 2.1FT} and the previous example that  
$$
\ctH(\sT, R)=\cH(\sT, R)\oplus R \tE. 
$$
\end{rem}


\subsection{Fourier expansion} 

The theory of Fourier expansions of automorphic forms over function fields was developed by Weil in \cite{Weil}. 
As was observed by P\'al in \cite{Pal}, Weil's theory works over more general rings than $\C$. 
Here we follow Gekeler's reinterpretation \cite{Improper} of  
Weil's adelic approach as analysis on the Bruhat-Tits tree, but we will extend \cite{Improper} to the 
setting of these more general rings. 

\begin{defn}
Following \cite{Pal} we say that $R$ is a \textit{coefficient ring} if $p\in R^\times$
and $R$ is a quotient of a discrete valuation ring $\tilde{R}$ which contains $p$-th roots of unity. Note 
that the image of the $p$-th roots of unity of $\tilde{R}$ in $R$ is exactly the set of $p$-th roots 
of unity of $R$. For example, any algebraically closed field of characteristic different from $p$ 
is a coefficient ring. 
\end{defn}

In this subsection $R$ is assumed to be a coefficient ring. Let 
\begin{align*}
\eta: \Fi &\to R^\times \\ 
\sum a_i\pi_\infty^i &\mapsto \eta_0\Big(\text{Trace}_{\F_q/\F_p}(a_1)\Big)
\end{align*}
where $\eta_0: \F_p\to R^\times$ is a non-trivial additive character fixed once and for all.  

The group 
$$
\G_\infty:=\left\{\begin{pmatrix} a & b \\ 0 & d\end{pmatrix}\in \GL_2(A)\right\}  
$$
acts orientation preserving on $\sT$, so $E(\G_\infty\bs \sT)=\G_\infty\bs E(\sT)$, and the 
orientation $E(\sT)^+$ on $\sT$ induces an orientation $E(\G_\infty\bs \sT)^+$. 
For an $R$-valued $\G_\infty$-invariant function $f$ on $E(\sT)$,  
its \textit{constant Fourier coefficient} is the $R$-valued function $f^0$ on $\pi_\infty^\Z$ defined by 
$$
f^0(\pi_\infty^k)= \begin{cases}
q^{1-k} \sum_{\substack{u\in (\pi_\infty)/(\pi_\infty^k)}}f\left(\begin{pmatrix}\pi_\infty^k & u \\ 0 &1\end{pmatrix}\right) 
& \text{if $k\geq 1$}, \\ 
f\left(\begin{pmatrix}\pi_\infty^k & 0 \\ 0 &1\end{pmatrix}\right)& \text{if $k\leq 1$}.  
\end{cases}
$$
For a divisor $\fm$ on $F$, the \textit{$\fm$-th Fourier coefficient} $f^\ast(\fm)$ of $f$ is 
$$
f^\ast(\fm) = q^{-1-\deg(\fm)}\sum_{u\in (\pi_\infty)/(\pi_\infty^{2+\deg(\fm)})} 
f\left(\begin{pmatrix}\pi_\infty^{2+\deg(\fm)} & u \\ 0 &1\end{pmatrix}\right)\eta(-m u), 
$$ 
if $\fm$ is non-negative, and $f^\ast(\fm)=0$, otherwise; here $m\in A$ is the monic polynomial 
such that $\fm=\mathrm{div}(m)\cdot \infty^{\deg(\fm)}$. 
Then $f$ has a \textit{Fourier expansion}
\begin{equation}\label{eq-Fexp}
f \left(\begin{pmatrix} \pi_\infty^k & y \\ 0 &1\end{pmatrix}\right) = f^0(\pi_\infty^k)+ 
\sum_{\substack{0\neq m\in A \\ \deg(m)\leq k-2}} f^\ast(\mathrm{div}(m)\cdot \infty^{k-2})\cdot \eta(my).   
\end{equation}
We refer to \cite[$\S$2]{Pal} and \cite[$\S$2]{Improper} for the proofs. 

\begin{lem}\label{lemFH} Assume $f\in \ctH(\sT, R)^{\G_\infty}$.  Then 
\begin{itemize}
\item[(i)] 
$f^0(\pi_\infty^k)=f^0(1)\cdot q^{-k}$ for any $k\in \Z$; 
\item[(ii)] $
f^\ast(\fm \infty^k)=f^\ast(\fm)\cdot q^{-k}$ for any non-negative divisor $\fm$ and $k \in \Z_{\geq 0}$. 
\end{itemize}
In particular, when $f$ is harmonic, the Fourier coefficients $f^0(1)$ and $f^*(\fm)$ for $\fm \lhd A$ uniquely determine $f$.
\end{lem}
\begin{proof}
Note that the pseudo-harmonicity of $f$ says:
$$ 
f|T_\infty (g) := \sum_{u \in \F_q} f\left(g \begin{pmatrix} \pi_\infty &u\\ 0&1\end{pmatrix}\right) = f(g), \quad \forall g \in \GL_2(F_\infty).
$$
It is then straightforward to check that for any non-negative divisor $\fm$ and $k \in \Z$, 
$$(f|T_\infty)^0(\pi_\infty^k) = q f^0(\pi_\infty^{k+1}) \quad \text{ and } \quad
(f|T_\infty)^*(\fm) = q f^*(\fm \infty).$$
Hence we obtain (i) and (ii).
The last statement follows from the Fourier expansion of $f$.
\end{proof}

\begin{lem}\label{lem2.3FT}
Let $\fn \lhd A$ be a non-zero ideal. 
Then $f \in \widetilde{\cH}(\fn,R)$ is uniquely determined by its Fourier coefficients $f^0(1)$ and $f^*(\fm)$ for $\fm \lhd A$.
\end{lem}

\begin{proof}
The Fourier coefficients of $f$ determine the values of $f$ on positively oriented edges $e \in E(\sT)^+$. 
Therefore it suffices to show that for every edge $e \in E(\sT)^-$, there exists $\gamma \in \Gamma_0(\fn)$ such that $\gamma e \in E(\sT)^+$.
By (\ref{eq-edge+}), after identifying $E(\sT)$ with $\GL_2(F_\infty)/Z(F_\infty) \cI_\infty$, every edge $e \in E(\sT)^-$ is expressed by
$$\begin{pmatrix} \pi_\infty^r & u \\ 0 & 1\end{pmatrix}\begin{pmatrix} 0 & 1 \\ \pi_\infty & 0 \end{pmatrix}, \quad r \in \Z,\ u \in F_\infty.$$
By the approximation theorem, there exists $\alpha \in F^{\times}$ such that 
\begin{align*} &\ord_\infty(u-\alpha) \geq r \quad \text{ and }\\
&\ord_\fp\left(\fp^{-\ord_\fp(\fn)} - \alpha\right) \geq 0\ \text{ for all prime } \fp \mid \fn.
\end{align*}
Writing $\alpha = d/c$ with $c,d \in A$ and $(c,d) = 1$, the above inequalities imply that
$$c = c_1 \fn \text{ with $c_1 \in A$,} \quad \text{ and } \quad \ord_\infty \left(u-\frac{d}{c_1 \fn}\right) \geq r.$$
Take $a,b \in A$ with $ad - b c_1 \fn =1$. Then
$$\gamma:= \begin{pmatrix} a&b\\c_1 \fn &d \end{pmatrix} \in \Gamma_0(\fn) \quad \text{ and }\quad  
\ord_\infty \left(\frac{ c_1 \fn u + d}{ c_1 \fn \pi_\infty^r}\right) \geq 0.$$
Therefore $\gamma e \in E(\sT)^+$ as
$$\begin{pmatrix} a&b \\ c_1 \fn & d\end{pmatrix} \cdot \begin{pmatrix} \pi_\infty^r & u \\ 0 & 1\end{pmatrix}\begin{pmatrix} 0 & 1 \\ \pi_\infty & 0 \end{pmatrix} \cdot \begin{pmatrix} 1 & 0 \\ - \left(\frac{ c_1 \fn u + d}{c_1 \fn u \pi_\infty ^r}\right) \pi_\infty & 1\end{pmatrix}
= \begin{pmatrix} -\frac{\pi_\infty}{c_1 \fn} & u \pi_\infty^r \\ 0 & c_1 \fn \pi_\infty^r \end{pmatrix}.$$
\end{proof}


\subsection{Hecke operators and Atkin-Lehner involutions} 

Given a non-zero ideal $\fm\lhd A$, let 
\begin{equation}
B_\fm=\begin{pmatrix} \fm & 0 \\  0 & 1\end{pmatrix},
\end{equation}
which we consider as a linear operator on $R$-valued functions on $E(\sT)$. 
Note that the action of $B_\fm^{-1}$ on functions on $E(\sT)$ is the same as the action of the matrix 
$\begin{pmatrix} 1 & 0 \\  0 & \fm\end{pmatrix}$ (since the diagonal matrices act trivially). 
For $b\in A$, let $S_b:=\begin{pmatrix} 1 & b \\ 0 & 1\end{pmatrix}$. For a prime ideal $\fp\lhd A$, 
define the \textit{Hecke operators} $U_\fp$ and $T_\fp$ acting on the space of $R$-valued functions on $E(\sT)$ by 
\begin{align*}
f|U_\fp &=\sum_{\substack{b\in A\\ \deg(b)<\deg(\fp)}} f|B_\fp^{-1}S_b, \\ f|T_\fp &=f|U_\fp + f|B_\fp. 
\end{align*}

For an ideal $\fm \parallel \fn$, let $W_\fm$ be any matrix of the form 
\begin{equation}\label{ALmatrix}
\begin{pmatrix} a\fm & b \\ c\fn & d\fm \end{pmatrix} 
\end{equation}
such that $a,b,c,d, \in A$, and the ideal generated by $\det(W_\fm)$ in $A$ is $\fm$. 
It is not hard to check that for $f\in \ctH(\fn, R)$, $f|W_\fm$ does not depend on the choice 
of the matrix for $W_\fm$. Moreover, 
as $R$-linear endomorphisms of $\ctH(\fn, R)$, $W_\fm$'s satisfy 
\begin{equation}\label{eqWs}
W_{\fm_1}W_{\fm_2}=W_{\fm_3}, \quad \text{where} \quad \fm_3=
\frac{\fm_1\fm_2}{(\fm_1, \fm_2)^2}. 
\end{equation}
Therefore, the matrices $W_\fm$ acting on $\ctH(\fn, R)$ 
generate an abelian group $\W\cong (\Z/2\Z)^s$, called the group of \textit{Atkin-Lehner involutions}, 
where $s$ is the number of prime divisors of $\fn$. 

\begin{lem}\label{lem-TU} Given a non-zero ideal $\fn\lhd A$, 
the operators $T_\fp$ $(\fp\nmid \fn)$, $U_\fp$ $(\fp\mid\fn)$, and $W_\fm$ $(\fm\parallel \fn)$ 
preserve the $R$-modules $\ctH(\fn, R)$, $\cH(\fn, R)$, $\cH_0(\fn, R)$. Moreover, $T_\fp U_{\fp'}=U_{\fp'}T_\fp$, 
$T_\fp W_\fm=W_\fm T_\fp$, and if $\fp\nmid \fm$ then $U_\fp W_\fm=W_\fm U_\fp$.  
\end{lem}
\begin{proof} The group-theoretic proofs of the analogous statement for operators 
acting on classical modular forms work also in this setting; cf. \cite[$\S$4.5]{Miyake}. 
\end{proof}

Given ideals $\fn, \fm\lhd A$, denote 
$$
\G_0(\fn,\fm)=\left\{\begin{pmatrix} a & b \\  c & d\end{pmatrix}\in \GL_2(A)\ \big|\ c\in \fn, b\in \fm \right\}. 
$$

\begin{lem}\label{lemAL2} We have:
\begin{enumerate}
\item If $f\in \ctH(\fn, R)$, then $f|B_\fm$ is $\G_0(\fn\fm)$-invariant and 
$f|B_\fm^{-1}$ is $\G_0(\fn/(\fn, \fm), \fm)$-invariant.  
\item
If $\fn$ is coprime to $\fm$ and $f\in \ctH(\fn, R)$, then 
$$(f|B_\fm)|W_\fm=f,$$ 
where $W_\fm$ is the Atkin-Lehner involution acting on $\ctH(\fn\fm, R)$. 
\item If $\fm\parallel \fn$, $(\fb, \fm)=1$, and $f\in \ctH(\fn, R)$, then 
$$
(f|B_\fb)|W_\fm=(f|W_\fm)|B_\fb,
$$ 
where on the left hand-side $W_\fm$ denotes the Atkin-Lehner involution 
acting on $\ctH(\fn\fb, R)$ and on the right hand-side $W_\fm$ denotes the  
involution acting on $\ctH(\fn, R)$. 
\item If $f\in \ctH(\sT, R)^{\G_\infty}$, then 
$$
(f|B_\fp)|U_\fp=|\fp|\cdot f. 
$$
\end{enumerate}
\end{lem}
\begin{proof} This follows from straightforward manipulations with matrices; cf. \cite[$\S$2]{AL}. 
\end{proof}

\begin{lem}\label{lem_new14} Assume $R$ is a coefficient ring. 
For any non-zero ideal $\fm\lhd A$ and $f\in\ctH(\sT, R)^{\G_\infty}$ we have 
\begin{align*}
(f|B_\fm)^0(\pi_\infty^k) &=f^0(\pi_\infty^{k-\deg(\fm)})\\ (f|B_\fm)^\ast(\fn) &=f^\ast(\fn/\fm). 
\end{align*}
\end{lem}
\begin{proof}
See Proposition 2.10 in \cite{Improper}. 
\end{proof}

\begin{lem}\label{lem1.15fm} Assume $R$ is a coefficient ring and $f\in \ctH(\fn, R)$ is an eigenfunction of all $T_\fp$, $\fp\nmid \fn$; that is, 
$f|T_\fp=\la_\fp f$ for some $\la_\fp\in R$. Then the Fourier 
coefficients $f^\ast(\fm)$, with $\fm$ coprime to $\fn$, are uniquely determined by $f^\ast(1)$ 
and the eigenvalues $\la_\fp$, $\fp\nmid \fn$. 
\end{lem}
\begin{proof} This follows from two well-known facts: (1) every Hecke operator $T_\fm$, with $\fm$ coprime to $\fn$, 
can be expressed as a polynomial in $T_\fp$'s with integer coefficients, where $\fp$'s are the prime divisors of $\fm$; 
(2) $(f|T_\fm)^\ast(1)=|\fm|f^\ast(\fm)$. For more details, see the discussion in \cite[$\S$3]{Analytical}. 
\end{proof}

As we mentioned, the cusps of $\G_0(\fn)$ are in bijection with the orbits 
of the action of $\G_0(\fn)$ on 
$$
\p^1(F)=\p^1(A)=\left\{\begin{pmatrix} a \\ b\end{pmatrix}\ \big|\ a, b\in A, \ (a, b)=1,\ a \text{ is monic}\right\}, 
$$
where $\G_0(\fn)$ acts on $\p^1(F)$ from the left as on column vectors. We leave the proof 
of the following lemma to the reader. 

\begin{lem}\label{lemCusps} Assume $\fn$ is square-free. 
\begin{enumerate}
\item 
For a monic $\fm|\fn$, let $[\fm]$ 
be the orbit of $\begin{pmatrix} 1 \\ \fm \end{pmatrix}$ under the action of $\G_0(\fn)$. 
Then $[\fm]\neq [\fm']$ if $\fm\neq \fm'$, and the set $\{[\fm]\ |\ \fm|\fn\}$ 
is the set of cusps of $\G_0(\fn)$. In particular, there are $2^s$ cusps, where $s$ 
is the number of prime divisors of $\fn$. 
\item Since $W_\fd$ normalizes $\G_0(\fn)$, 
it acts on the set of cusps of $\G_0(\fn)$. We have  
$$
W_\fd [\fm] = \left[\frac{\fm\fd}{(\fm, \fd)^2}\right]. 
$$
\end{enumerate}
\end{lem}

\begin{defn} Let $\fn=\fp_1\cdots \fp_s$ be a square-free ideal in $A$ with the given prime decomposition. 
Let 
$$
\E=\{(\eps_1, \dots, \eps_s)\ |\ \eps_i=\pm 1,\ 1\leq i\leq s\}.
$$ 
Let $R$ be a $\Z[1/2]$-algebra and $M$ be an $R[\W]$-module. For each $\beps\in \E$, we let 
$M^\beps$ be the maximum direct summand of $M$ on which $W_{\fp_i}$ acts by multiplication by $\eps_i$ ($1\leq i\leq s$), so that 
$$
M=\bigoplus_{\beps\in \E} M^\beps. 
$$
\end{defn}


\subsection{Atkin-Lehner type result} The results in this subsection are the analogues for $R$-valued 
pseudo-harmonic cochains of some of the results of Atkin and Lehner \cite{AL}. In our proofs 
we also use ideas of Ohta from \cite[$\S$2.1]{Ohta}. 

\begin{lem}\label{lemLevelLow} Assume $R$ is a coefficient ring. Let $\fp\lhd A$ be a prime such that $\fp\parallel \fn$. 
Suppose $f\in \ctH(\fn, R)$ is such that $f^\ast(\fm)=0$ if $\fp\nmid \fm$. 
Then there is $g\in \ctH(\fn/\fp, R)$ such that 
$$f=g|B_\fp\quad \text{ and }\quad f|W_\fp= g.$$ 
Moreover, if $f$ is harmonic, then so is $g$. 
\end{lem}
\begin{proof} Let $g:=f|B_\fp^{-1}$. By Lemma \ref{lemAL2} (1), $g$ is $\G_0(\fn/\fp, \fp)$-invariant and pseudo-harmonic. 
Moreover, given $a \in A$ and $e = \begin{pmatrix} \pi_\infty^k & u \\ 0 & 1 \end{pmatrix} \in E(\sT)^+$,
\begin{eqnarray}
g\left(\begin{pmatrix}1&a\\0&1\end{pmatrix} e\right) & = & f \left(\begin{pmatrix} \pi_\infty^{k + \deg \fp} & u + \frac{a}{\fp} \\ 0 & 1 \end{pmatrix}\right) \nonumber \\
& = &  f^0 (\pi_\infty^{k + \deg \fp}) + \sum_{\fm \lhd A\atop \deg \fm \leq k+\deg \fp-2 } f^*(\text{div}(m) \infty^{k+\deg \fp - 2}) \eta(mu + \frac{ma}{\fp}) \nonumber \\
& = & f^0 (\pi_\infty^{k + \deg \fp}) + \sum_{\fm \lhd A\atop \deg \fm \leq k+\deg \fp-2 } f^*(\text{div}(m) \infty^{k+\deg \fp - 2}) \eta(mu) \nonumber \\
& = & g(e). \nonumber
\end{eqnarray}
The third equality in the above follows from the assumption that $f^\ast(\fm) = 0$ unless $\fp \mid \fm$ and $\eta(\alpha) = 1$ for every $\alpha \in A$.
Moreover, by Lemma~\ref{lem 2.1FT}, there exists a constant $c \in R$ such that
$$g\left(\begin{pmatrix}1&a\\0&1\end{pmatrix} \bar{e}\right) = c+g\left(\begin{pmatrix}1&a\\0&1\end{pmatrix} e\right) = c + g(e) = g(\bar{e}).$$
Thus $g$ is also $\G_\infty$-invariant.
Since the subgroup of $\GL_2(A)$ generated by $\G_0(\fn/\fp, \fp)$ and $\G_\infty$ is $\G_0(\fn/\fp)$, 
we conclude that $g\in \ctH(\fn/\fp, R)$ (resp. $g\in \cH(\fn/\fp, R)$ if $f$ is harmonic). 

Now choose some matrix $\begin{pmatrix} a\fp & b \\ \fn & d\fp\end{pmatrix}$ representing $W_\fp$. 
Since 
$$
B_\fp W_\fp 
=\begin{pmatrix} a\fp^2 & b\fp \\ \fn & d\fp\end{pmatrix}= 
\begin{pmatrix} \fp & 0 \\ 0 &\fp\end{pmatrix} \begin{pmatrix} a\fp & b \\ \fn/\fp & d\end{pmatrix}
\in Z(\Fi)\G_0(\fn/\fp), 
$$
we have $f|W_\fp=g|B_\fp W_\fp =g$. 
\end{proof}

\begin{prop}\label{propALOhta}
Let $\fn=\fp_1\cdots \fp_s\lhd A$ be square-free, and $R$ a coefficient ring. Assume $f\in \ctH(\fn, R)$ is such that $f^\ast(\fm)=0$ 
unless $\fm$ is divisible by some $\fp_i$, $1\leq i\leq s$. Assume $f|W_{\fp_s}=\pm f$. Then 
$f^\ast(\fm)=0$ unless $\fm$ is divisible by some $\fp_i$, $1\leq i\leq s-1$, when $s\geq 2$, and $f^\ast(\fm)=0$ for all $\fm$  
when $s=1$. 
\end{prop}
\begin{proof} Let $\fp \mid \fn$. Define the analogue of the ``annihilator'' operator of Atkin and Lehner \cite{AL}:
$$
K_\fp:=1-\frac{1}{|\fp|}U_\fp B_\fp. 
$$
Since $f|U_\fp\in \ctH(\fn, R)$ 
satisfies 
$
(f|U_\fp)^\ast(\fm)=|\fp|f^\ast(\fm\fp)$  
and $f|B_\fp\in \ctH(\fn\fp, R)$ satisfies $(f|B_\fp)^\ast(\fm)=f^\ast(\fm/\fp)$, $f|K_\fp$ is in $\ctH(\fn\fp, R)$ 
and has Fourier coefficients 
$$
(f|K_\fp)^\ast(\fm) = 
\begin{cases}
f^\ast(\fm) & \text{if $\fp\nmid \fm$}; \\
0 & \text{if $\fp \mid \fm$}.
\end{cases}
$$

Given two distinct primes $\fp, \fp'$ dividing $\fn$,  $K_\fp$ commutes with $K_{\fp'}$ and $W_{\fp'}$; cf. 
\cite[Lem. 2.23]{PW1}.  
Set $h:=f|\prod_{i=1}^{s-1}K_{\fp_i} \in \ctH(\fn\fp_1\cdots\fp_{s-1}, R)$. 
Then $h^\ast(\fm)=0$ unless $\fp_s \mid \fm$, so Lemma \ref{lemLevelLow} implies that there is $h'\in \ctH(\fn\fp_1\cdots\fp_{s-1}/\fp_s, R)$ 
such that $h=h'|B_{\fp_s}$ and $h|W_{\fp_s}=h'$. The equality $h=h'|B_{\fp_s}$ implies that $h^\ast(\fp_s\fm)={h'}^\ast(\fm)$ for all $\fm$. 
On the other hand, $h|W_{\fp_s}=h'$ implies 
$$
h'=h|W_{\fp_s}=f|\prod_{i=1}^{s-1}K_{\fp_i} W_{\fp_s} = f|W_{\fp_s}\prod_{i=1}^{s-1}K_{\fp_i} =\pm f|\prod_{i=1}^{s-1}K_{\fp_i} = \pm h. 
$$
Thus $h^\ast(\fp_s\fm)={h'}^\ast(\fm)=\pm h^\ast(\fm)$ for all $\fm$. 
Now it is easy to see that $h^\ast(\fm)=0$ for all $\fm$,  
and therefore $f^\ast(\fm)=0$ unless $\fm$ is divisible by some $\fp_i$, $1\leq i\leq s-1$. 
\end{proof}


\section{Eisenstein pseudo-harmonic cochains}\label{sEPHC}

Let $\fn\lhd A$ be a non-zero ideal. 
We say that $f \in \ctH(\fn,R)$ is \textit{Eisenstein} 
if $f |T_{\fp} = (|\fp|+1)f$ for every prime ideal $\fp \lhd A$ not dividing $\fn$. 
It is clear that Eisenstein pseudo-harmonic cochains form an $R$-submodule of $\ctH(\fn, R)$ 
which we denote by $\ctE(\fn, R)$. Let $\cE(\fn, R)$ (resp. $\cE_0(\fn, R)$, $\cE_{00}(\fn, R)$) 
be the intersection of $ \ctE(\fn, R)$ with $\cH(\fn, R)$ (resp. $\cH_0(\fn, R)$, $\cH_{00}(\fn, R)$).  
We have 
$$
\ctE(\fn, R)\supset \cE(\fn, R) \supset \cE_0(\fn, R) \supset \cE_{00}(\fn, R). 
$$
The main goal of this section is to determine $\cE_0(\fn, R)$ for square-free $\fn$. For that purpose, using $\tE$ from Example \ref{example1}, 
we will construct certain explicit Eisenstein pseudo-harmonic cochains. 

\begin{lem}\label{lemtEEis}
For any prime $\fp\lhd A$, we have $\tE|T_\fp=(1+|\fp|)\tE$. Therefore, $\ctE(1, R)=\ctH(1, R)$. 
\end{lem}
\begin{proof}
By Lemma \ref{lem-TU}, $\tE|T_\fp\in \ctH(1, R)$, and since $\tE$ spans $\ctH(1, R)$,    
we have $\tE|T_\fp = c \tE$ for some $c\in R$. To find $c$, it is enough to compute the value $\tE|T_\fp$ 
on some $e_i$; it is convenient to take $i=\deg(\fp)$:
$$
\tE|T_\fp (e_{\deg(\fp)})= \tE\left(\begin{pmatrix} \fp^2 & 0 \\ 0 & 1\end{pmatrix}\right) 
+\sum_{\deg(b)<\deg(\fp)} \tE\left(\begin{pmatrix} 1 & b \\ 0 & \fp\end{pmatrix}\begin{pmatrix} \fp & 0 \\ 0 & 1\end{pmatrix}\right)
$$
$$
=\tE(e_{2\deg(\fp)}) + |\fp| \tE(e_{0})= q^{2\deg(\fp)+1}+q^{\deg(\fp)+1}=(|\fp|+1)\tE(e_{\deg(\fp)}). 
$$
\end{proof}

We will need to know the Fourier coefficient of $\tE$, which can be computed as follows. First, $$\tE^0(1)=\tE(e_0)=q.$$ Next, 
using (\ref{eq-Fexp}) and Lemma \ref{lemFH}, we compute
\begin{align*}
\tE^\ast(1) & =-\left(\tE\left(\begin{pmatrix} \pi_\infty^2 & \pi_\infty \\ 0 & 1\end{pmatrix}\right)- \tE^0(\pi_\infty^2)\right)\\ 
&=
-\left(\tE\left(\begin{pmatrix} \pi_\infty^2 & \pi_\infty \\ 0 & 1\end{pmatrix}\right)
- q^{-2}\tE\left(\begin{pmatrix} 1 & 0 \\ 0 & 1\end{pmatrix}\right)\right) \\
&= -\left(\tE(e_0)- q^{-2}\tE(e_0)\right) \\
&= \frac{1-q^2}{q}. 
\end{align*}
Finally, since $\tE$ is an eigenfunction of $T_\fp$ with eigenvalue $|\fp|+1$ for all prime $\fp$, any $f^\ast(\fm)$ can be computed from $f^\ast(1)$; 
see Lemma \ref{lem1.15fm}. We get 
for $k \in \Z$ and $u \in F_\infty$,
$$\widetilde{E}\begin{pmatrix} \pi_\infty^k & u \\ 0 & 1\end{pmatrix}
= q^{-k+1}\left(1 + (1-q^2)\sum_{m \in A, \atop \deg m + 2 \leq k} \sigma(m) \eta(m u)\right),$$
where $$\sigma(m) := \sum_{\text{monic } m' \in A, \atop m' \mid m} |m'|.$$

\begin{rem}\label{remH}
In \cite{Improper}, Gekeler introduced the so-called improper Eisenstein series $H$ on $\sT$, 
as the analogue of the classical Eisenstein series $E_2(z)=\sum_{c,d\in \Z}'(cz+d)^{-2}$. 
This function $H$ is a $\G_\infty$-invariant 
harmonic cochain on $E(\sT)$, which is ``half'' invariant under $\GL_2(A)$ in the sense 
that $H(\gamma e)=H(e)$ if and only if both $e$ and $\gamma e$ are positively or negatively oriented. 
It can be seen from \cite[p. 383]{Improper} that $\tE(e)=q\cdot H(e)$ for all $e\in E(\sT)^+$, although these functions 
are distinct (for example, $\tE$ is $\GL_2(A)$-invariant and is not alternating).  
\end{rem}

\begin{notn}\label{notn2.3}
From now on $\fn=\fp_1\cdots \fp_s$ is a square-free ideal with the given prime decomposition. Let 
\begin{align*}
\beps_{H(\fn)} &:=((-1)^{\deg(\fp_1)}, \dots, (-1)^{\deg(\fp_s)}) \in \E, \\ 
\bone &:=(1,1,\dots, 1)\in \E.
\end{align*}
Also, for a given $\beps=(\eps_1, \dots, \eps_n)\in \E$ and a divisor $\fd \mid \fn$, put  
$$
\epsilon_\fd = \prod_{\fp_i\mid \fd}\eps_i\in \{\pm 1\}.
$$
(In particular, $\epsilon_{\fp_i}=\eps_i$ and $\epsilon_1=1$.) 
\end{notn}

\begin{prop}\label{prop-cor2.4}  
Let $R$ be a coefficient ring with $(q-1)\in R^\times$. Given $\beps\in \E$, we have 
$$
\{f\in \ctH(\fn, R)^\beps\ |\ f^\ast(\fm)=0 \text{ when } (\fm, \fn)=1\}=
\begin{cases}
\cH(1, R) & \text{if }\beps=\beps_{H(\fn)},\\
0 & \text{otherwise}. 
\end{cases}
$$
\end{prop}
\begin{proof} 
Let $f\in \ctH(\fn, R)^\beps$ and assume $f^\ast(\fm)=0$ for all $\fm$ coprime to $\fn$. Then Proposition \ref{propALOhta} 
implies that $f^\ast(\fm)=0$ for all $\fm$. In particular, $f^\ast(1)=0$. 
Applying Lemma \ref{lemLevelLow} successively we get that there is $g\in \ctH(1, R)$ 
such that $f=g|\prod_{i=1}^s B_{\fp_i}$ and $\epsilon_\fn f=f|W_\fn = g\in \ctH(1, R)$. Therefore, 
there exists $a\in R$ such that $f=a\tE$. This implies 
$$
f^\ast(1)=a\tE^\ast(1)=\frac{a(1-q^2)}{q}=0. 
$$
Since $(q-1)$ is invertible in  $R$, we must have $a\in R[q+1]$. From Example \ref{example1} we conclude that $f\in \cH(1, R)$ is harmonic. 
Now it is enough to show that $a\tE|W_{\fp_i}=(-1)^{\deg(\fp_i)}a\tE$ for all $1\leq i\leq s$ if $a\in R[q+1]$. 
For that, let $\fp|\fn$ be a prime. By Lemma \ref{lemAL2} (2), $a\tE|B_\fp W_\fp=a\tE$. On the other hand, 
\begin{align*}
a\tE|B_\fp\left(\begin{pmatrix}\pi_\infty^k & u \\ 0 &1\end{pmatrix}\right) 
& = a\tE \left(\begin{pmatrix}\pi_\infty^{k-\deg(\fp)} & u \fp \\ 0 &1\end{pmatrix}\right)=aq^{k-\deg(\fp)+1}=a(-1)^{k-\deg(\fp)+1} \\
&= (-1)^{\deg(\fp)} a\tE \left(\begin{pmatrix}\pi_\infty^k & u \\ 0 &1\end{pmatrix}\right). 
\end{align*}
Hence $a\tE|W_\fp= (-1)^{\deg(\fp)}a\tE|B_\fp W_\fp=(-1)^{\deg(\fp)}a\tE$, as was required to be shown.  
\end{proof}

\begin{notn} For $\beps \in \E$, let 
\begin{align*}
E^\beps &:=\frac{1}{\nu(\beps)}\tE|\prod_{i=1}^s(1+\eps_i W_{\fp_i})=\frac{1}{\nu(\beps)}\tE|\prod_{i=1}^s(1+\eps_i B_{\fp_i})\\ 
&=\frac{1}{\nu(\beps)}\sum_{\fd|\fn}\epsilon_\fd \tE|B_\fd \in \ctE(\fn, \Q),  
\end{align*}
where 
$$\nu (\beps) = \begin{cases}
1 & \text{ if $\beps = \beps_{H(\fn)}$, } \\
q+1 & \text{ otherwise.}
\end{cases}$$
(The fact that $E^\beps$ is Eisenstein follows from Lemma \ref{lemtEEis} and the commutativity $B_\fd T_\fp=T_\fp B_\fd$, $\fp\nmid \fd$.) 
Let
$$
N(\fn, \beps):=\prod_{i=1}^s(1+\eps_i|\fp_i|). 
$$
\end{notn}
\begin{rem}\label{remNnu}
Note that $q+1$ divides $N(\fn, \beps)$ if $\beps\neq \beps_{H(\fn)}$. In particular, $N(\fn, \beps)/\nu(\beps)\in \Z$. 
Also note that $(q+1, N(\fn, \beps_{H(\fn)})) = (q+1,2^s) = 2^b$ for some $0\leq b\leq s$, since
\begin{align*}
N(\fn, \beps_{H(\fn)}) &= \prod_{i=1}^s(1+(-1)^{\deg(\fp_i)}|\fp_i|)\\ & \equiv \prod_{i=1}^s(1+(-1)^{\deg(\fp_i)}(-1)^{\deg(\fp_i)})=2^s\ (\mod\ q+1). 
\end{align*}
\end{rem}

\begin{lem}\label{lemEeps}
We have $E^\beps\in \ctE(\fn, \Z)^\beps$. Moreover, $E^\beps$ is harmonic if and only if $\beps\neq \bone$. 
\end{lem}
\begin{proof}
Since $\tE$ is $\Z$-valued, $E^\beps$ takes values in $\Z[\nu(\beps)^{-1}]$. On the other hand, from the 
Fourier expansion of $\tE$ and Lemma \ref{lem_new14}, one computes 
\begin{equation}\label{eqE(1)}
(E^{\beps})^0(1) = q \frac{N(\fn,\beps)}{\nu(\beps)} \quad \text{ and } \quad (E^{\beps})^*(1) = \frac{1-q^2}{q\nu(\beps)},
\end{equation}
and also that $(E^{\beps})^*(\fm)\in \Z[p^{-1}]$ for all $\fm$. The Fourier expansion (\ref{eq-Fexp}) then implies that $E^\beps$ on $E(\sT)^+$
takes values in $\Z[p^{-1}]$; cf. \cite[Cor. 3.15]{Analytical}. 

Next, using the definition of $E^\beps$ and Example \ref{example1}, we have 
\begin{align*}
E^\beps(e_0) &=\frac{1}{\nu(\beps)}\sum_{\fd|\fn}\epsilon_\fd \tE|B_\fd(e_0) \\ &= \frac{1}{\nu(\beps)}\sum_{\fd|\fn}\epsilon_\fd \tE(e_{\deg(\fd)}) 
= \frac{q}{\nu(\fn)}\sum_{\fd|\fn}\epsilon_\fd |\fd|=\frac{q}{\nu(\fn)}N(\fn, \beps),  
\end{align*}
and similarly 
$$
E^\beps(\bar{e}_0)=\frac{q+1}{\nu(\fn)}\sum_{\fd|\fn}\epsilon_\fd - \frac{q}{\nu(\fn)}N(\fn, \beps) = \frac{q+1}{\nu(\beps)}\prod_{i=1}^s(1+\eps_{i}) 
- \frac{q}{\nu(\fn)}N(\fn, \beps). 
$$
Therefore, by Lemma \ref{lem 2.1FT}, for any $e\in E(\sT)$ we have 
\begin{equation}\label{eqE+E}
E^{\beps}(e) + E^{\beps}(\bar{e})= E^{\beps}(e_0) + E^{\beps}(\bar{e}_0) = \frac{q+1}{\nu(\beps)}\prod_{i=1}^s(1+\eps_{i}).
\end{equation}
This shows that $E^\beps$ is harmonic if and only if $\beps\neq \bone$, and that $E^\beps(e)\in \Z[p^{-1}]$ 
for all $e\in E(\sT)$. Since $p$ and $\nu(\beps)$ are coprime, we conclude that $E^\beps$ takes its values in $\Z$. 

It remains to show that $E^\beps|W_{\fp_i}=\eps_i E^\beps$, $1\leq i\leq s$. This follows from the definition of $E^\beps$ 
and the following observation: 
$$
\prod_{j=1}^s(1+\eps_j W_{\fp_j})W_{\fp_i} = (W_{\fp_i}+\eps_i)\prod_{j\neq i}(1+\eps_j W_{\fp_j})=\eps_i \prod_{j=1}^s(1+\eps_j W_{\fp_j}). 
$$
\end{proof}

The previous lemma allows us to consider $E^\beps$ as an element of $\tE(\fn, R)^\beps$ for any ring $R$. 

\begin{prop}\label{prop 3.2FT}
Let $R$ be a coefficient ring with $(q-1) \in R^\times$. Assume $\beps \neq \beps_{H(\fn)}$. Then
$$\ctE(\fn,R)^{\beps} = R\cdot E^{\beps},$$
and 
$$\cE_0(\fn,R)^{\beps} = 
\begin{cases}
R\left[\frac{N(\fn,\beps)}{q+1}\right] \cdot E^{\beps}, & \text{if $\beps\neq \bone$};\\ 
0, & \text{if $\beps=\bone$}. 
\end{cases}
$$
\end{prop}

\begin{proof}
Since $\beps \neq \beps_{H(\fn)}$, from (\ref{eqE(1)}) we get $(E^{\beps})^*(1) = q^{-1}(1-q) \in R^{\times}$.
Given $f \in \widetilde{\cE}(\fn,R)^{\beps}$, let   
$$ \tilde{f}:= f - \frac{qf^*(1)}{1-q}E^{\beps}.$$
Then $\tilde{f} \in \widetilde{\cH}(\fn,R)^{\beps}$ has $\tilde{f}^*(\fm) = 0$ for every $\fm$ coprime to $\fn$. 
From Proposition \ref{prop-cor2.4} we get $\tilde{f} = 0$, i.e.,
\begin{equation}\label{eqSpanE}
f = \frac{q f^*(1)}{1-q}E^{\beps}.
\end{equation}
Hence $\ctE(\fn, R)^\beps$ is spanned by $E^\beps$. 

To prove the second claim, assume $f\in \cE_0(\fn,R)^{\beps}$. In particular, $f$ is alternating, so from (\ref{eqSpanE}) and (\ref{eqE+E}) we get 
$$
0=f(\bar{e}) + f(e) = \frac{q}{1-q}\left(\prod_{i=1}^s(1+\epsilon_{\fp_i})\right) \cdot f^*(1).
$$
Thus, when $\beps =\bone$, we must have $f^*(1) = 0$; thus, by Proposition \ref{prop-cor2.4}, $f = 0$. 
(Note that $\prod_{i=1}^s(1+\eps_i)=2^s\in R^\times$, since $p$ and $q-1$ are invertible in $R$, and one of these numbers is necessarily even.) 

Finally, assume $\beps\neq \bone$.  
According to \cite[Lem. 2.19]{PW1}, $f\in \cH(\fn,R)$ is cuspidal if and only if $(f|W)^0(1)=0$ for all $W\in \W$. 
Since our $f$ is an eigenfunction of all Atkin-Lehner involutions, we conclude that $f$ is cuspidal if and only if $f^0(1)=0$. 
On the other hand, from (\ref{eqE(1)}) and (\ref{eqSpanE}) we have 
$$f^0(1) = \frac{q^2}{1-q} \cdot \frac{N(\fn,\beps)}{q+1} \cdot f^*(1).$$ 
Since $q(q-1)\in R^\times$, we must have $f^*(1) \in R\left[\frac{N(\fn,\beps)}{q+1}\right]$; thus
$\cE_0(\fn,R)^{\beps} = 
R\left[\frac{N(\fn,\beps)}{q+1}\right] \cdot E^{\beps}$. 
\end{proof}

\begin{rem} \label{rem E_00=E_0}
Assume $R$ is a coefficient ring with $(q-1)\in R^\times$.
\begin{itemize}
\item[(1)] Take $\beps \in \E$ with $\beps\neq \beps_{H(\fn)}$. By (\ref{eqE+E}), $E^\beps\in \ctH(\fn, R)$ is harmonic if and only if $\beps\neq \bone$. Thus, Proposition \ref{prop 3.2FT} implies that 
$\ctE(\fn, R)^\beps=\cE(\fn, R)^\beps$ when $\beps\neq \bone$,  and $\cE(\fn, R)^\beps=0$ when $\beps=\bone$.
\item[(2)] Suppose $\beps_{H(\fn)} \neq \bone$, i.e., there is a prime factor of $\fn$ with odd degree. Then by \cite[Lemma 2.7 (1)]{PW1} we have $\cE_0(\fn, R)^\beps=\cE_{00}(\fn, R)^\beps$ for every $\beps \neq \beps_{H(\fn)}$. 
\end{itemize}
\end{rem}

When $\beps = \beps_{H(\fn)}$, $(E^{\beps})^*(1)$ is not invertible in $R$, so the argument in the proof of Proposition~\ref{prop 3.2FT} does not 
work. We will deal with this case in three separate propositions, after proving the following lemma: 

\begin{lem}\label{lem 3.8}
Define $\beps_{H(\fn),s} = (\eps_1,\dots ,\eps_s) \in \E$ by 
$$
\eps_i =\begin{cases}
(-1)^{\deg \fp_i} & 1\leq i<s;\\ 
 -(-1)^{\deg \fp_s} & i=s.
\end{cases}
$$
\begin{itemize}
\item[(1)] Suppose $\deg \fp_s$ is even. Then $E^{\beps_{H(\fn),s}} | U_{\fp_s} = E^{\beps_{H(\fn),s}}$.
\item[(2)] Suppose $\deg \fp_s$ is odd.  Let $n$ be an integer dividing both $q+1$ and $\deg \fp_s$. Let $R$ be a coefficient ring.  
Then for an arbitrary $a \in R[n]$, we have 
$$a \cdot E^{\beps_{H(\fn),s}} | U_{\fp_s} = -a \cdot E^{\beps_{H(\fn),s}}.$$
\end{itemize}
\end{lem}

\begin{proof}
We can write 
$$E^{\beps_{H(\fn),s}} = \frac{1}{q+1} \cdot E^{\beps_{H(\fn/\fp_s)}} | (1 -(-1)^{\deg \fp_s} B_{\fp_s}).$$
By Lemma \ref{lemAL2} (4), $E^{\beps_{H(\fn/\fp_s)}} | B_{\fp_s}  U_{\fp_s} = |\fp_s| \cdot E^{\beps_{H(\fn/\fp_s)}}$ and 
$$(|\fp_s|+1) E^{\beps_{H(\fn/\fp_s)}} = E^{\beps_{H(\fn/\fp_s)}} | T_{\fp_s} = E^{\beps_{H(\fn/\fp_s)}} | (U_{\fp_s} + B_{\fp_s}).$$
One then gets:
\begin{align}
E^{\beps_{H(\fn),s}} | U_{\fp_s} &= \frac{1}{q+1} \cdot E^{\beps_{H(\fn/\fp_s)}}|(U_{\fp_s} -(-1)^{\deg \fp_s} B_{\fp_s} U_{\fp_s}) \notag \\
&= \frac{1}{q+1} \cdot E^{\beps_{H(\fn/\fp_s)}} | (T_{\fp_s}-  B_{\fp_s}-(-1)^{\deg \fp_s} |\fp_s|) \notag \\
&= \left((1-(-1)^{\deg \fp_s}) \cdot \frac{|\fp_s|+1}{q+1}\right) \cdot E^{\beps_{H(\fn/\fp_s)}} +(-1)^{\deg \fp_s}\cdot E^{\beps_{H(\fn),s}}. \notag
\end{align}
Part (1) immediately follows from this; part (2) follows from the observation that if $q\equiv -1\ (\mod\ n)$ and $\deg(\fp_s)$ is odd, then 
$$
\frac{|\fp_s|+1}{q+1} =\sum_{i=0}^{\deg(\fp_s)-1}(-1)^i q^i\equiv \deg(\fp_s)\ (\mod\ n).
$$
\end{proof}

\begin{prop}\label{prop 3.4}
Let $R$ be a coefficient ring with $(q-1) \in R^\times$. Assume there exists at most one prime divisor $\fp_i$ of $\fn$ 
with $\deg \fp_i$ odd. Then 
$$\widetilde{\cE}(\fn,R)^{\beps_{H(\fn)}} = R \cdot E^{\beps_{H(\fn)}},$$
and 
$$\cE_0(\fn,R)^{\beps_{H(\fn)}} =
\begin{cases}
R\left[N(\fn,\beps_{H(\fn)})\right]\cdot E^{\beps_{H(\fn)}}, & \text{ if $\beps_{H(\fn)} \neq \mathbf{1}$}; \\
0, & \text{ if $\beps_{H(\fn)} = \mathbf{1}$.}
\end{cases}
$$
\end{prop}

\begin{proof}
When $s =1$, i.e., $\fn = \fp$ is a prime, take $\beps' = (-(-1)^{\deg \fp}) \in \E$. 
Then $(E^{\beps'})^*(1) = q^{-1}(1-q) \in R^{\times}$.
For $f \in \widetilde{\cE}(\fp,R)^{\beps_{H(\fp)}}$, put
$$\tilde{f} := f - \frac{q f^*(1)}{1-q} E^{\beps'}.$$
Then by Proposition~\ref{propALOhta}, 
$\tilde{f}^*(\fm) = 0$ unless $\fp \mid \fm$. Hence by Lemma~\ref{lemLevelLow}, 
there exists $g \in \widetilde{\cH}(1,R) = R \widetilde{E}$ such that
$\tilde{f} = g|B_\fp$ and 
$$\tilde{f}|W_\fp = (-1)^{\deg \fp} \left(f+\frac{q f^*(1)}{1-q} E^{\beps'}\right) = g.$$
In particular,
$$g+(-1)^{\deg \fp} g|B_\fp = g+(-1)^{\deg \fp}\tilde{f} = 2(-1)^{\deg \fp} f.$$
Writing $g = a \widetilde{E}$ for $a \in R$, we get 
$$g+(-1)^{\deg \fp} g|B_\fp = a E^{\beps_{H(\fp)}}$$
and so
$f = (2^{-1} (-1)^{\deg \fp} a) \cdot E^{\beps_{H(\fp)}}$.
This shows $$\widetilde{\cE}(\fp,R)^{\beps_{H(\fp)}} = R \cdot E^{\beps_{H(\fp)}}.$$

Now suppose $s>1$.
Without loss of generality, we assume that $\deg \fp_i$ is even for $2\leq i \leq s$.
Let  $\beps_{H(\fn),s} = (\varepsilon_1,\dots ,\varepsilon_s)\in \E$ be the element 
defined in Lemma~\ref{lem 3.8}. 
In particular, $\beps_{H(\fn),s} \neq \mathbf{1}$ and $\beps_{H(\fn),s} \neq\beps_{H(\fn)}$.
Given $f \in \widetilde{\cE}(\fn,R)^{\beps_{H(\fn)}}$, put
$$\tilde{f} := f - \frac{qf^*(1)}{1-q} E^{\beps_{H(\fn),s}}.$$
Then, by Proposition \ref{propALOhta}, $\tilde{f}^*(\fm) = 0$ unless $\fp_{s} | \fm$. By Lemma~\ref{lemLevelLow}, 
there exists $g \in \widetilde{\cH}(\fn/\fp_{s},R)$ such that
$\tilde{f} = g|B_{\fp_{s}}$ and 
$$\tilde{f}|W_{\fp_{s}} = f + \frac{q f^*(1)}{1-q} E^{\beps_{H(\fn),s}} = g.$$ 
In particular, 
$$g+g|B_{\fp_{s}} = 2 f \quad \text{ and } \quad 
g- g|B_{\fp_{s}} = g- g|W_{\fp_{s}} = \frac{2q f^*(1)}{1-q} \cdot E^{\beps_{H(\fp),s}}.$$

Consider the trace map : $\text{Tr}^{\fn}_{\fn/\fp_s}: \widetilde{\cH}(\fn,R) \rightarrow \widetilde{\cH}(\fn/\fp_s,R)$ defined by:
$$\text{Tr}^{\fn}_{\fn/\fp_s}(h):= h + h | W_{\fp_s} U_{\fp_s}.$$
Then Lemma~\ref{lem 3.8} (1) implies that 
$$\text{Tr}^{\fn}_{\fn/\fp_s}(E^{\beps_{H(\fn),s}}) = E^{\beps_{H(\fn),s}} - E^{\beps_{H(\fn),s}} | U_{\fp_s} = 0.$$
Therefore 
$$\text{Tr}^{\fn}_{\fn/\fp_s}(g) = \text{Tr}^{\fn}_{\fn/\fp_s}(g |W_{\fp_s}).$$
The left hand side is equal to $(|\fp_s|+1)g$ as $g$ is of level $\fn/\fp_s$, and the right hand side is nothing but $g| T_{\fp_s}$. 
This implies that $g \in \widetilde{\cE}(\fn/\fp_s,R)^{\beps_{H(\fn/\fp_s)}}$.
By induction, there exists $a \in R$ so that $g = a E^{\beps_{H(\fn/\fp_s)}}$,
and hence
$$f = 2^{-1}\left(g+ g|B_{\fp_s}\right) = 2^{-1} a \cdot E^{\beps_{H(\fn)}}.$$
Therefore $\widetilde{\cE}(\fn,R)^{\beps_{H(\fn)}} = R \cdot E^{\beps_{H(\fn)}}$, which proves the first statement 
of the proposition. The second statement now can be deduced by an argument similar to the argument in the proof of Proposition \ref{prop 3.2FT}. 
\end{proof}

\begin{prop}\label{prop 3.5}
Let $R$ be a coefficient ring with $(q-1) \in R^\times$. Assume there are at least two prime factors of $\fn$ with odd degree, and  
$(q+1, \deg \fp_1,...,\deg \fp_s)$ is invertible in $R$.
Then 
$$\cE_0(\fn,R)^{\beps_{H(\fn)}} =
R\left[N(\fn,\beps_{H(\fn)})\right]\cdot E^{\beps_{H(\fn)}}.$$
\end{prop}

\begin{proof}
For $f \in \cE_0(\fn,R)^{\beps_{H(\fn)}}$,
put 
$$\tilde{f} := (q+1) f - \frac{q f^*(1)}{1-q} E^{\beps_{H(\fn)}} \in \widetilde{\cE}(\fn,R)^{\beps_{H(\fn)}}.$$
Then $\tilde{f}^*(\fm) = 0$ for every $\fm$ coprime to $\fn$.
By Proposition \ref{prop-cor2.4}, we get $\tilde{f} \in \cH(1,R) = R[q+1] \cdot \widetilde{E}$.
In particular, 
$$(q+1)^2f = \frac{q (q+1) f^*(1)}{1-q} E^{\beps_{H(\fn)}}.$$
As $f$ is cuspidal, we get 
\begin{align}\label{eqn 3.1}
(q+1)^2 f^0(1) &= \frac{q (q+1) f^*(1)}{1-q} (E^{\beps_{H(\fn)}})^0(1) \\ 
\nonumber &= \frac{q}{1-q} \cdot (q+1)N(\fn,\beps_{H(\fn)}) f^*(1) = 0.
\end{align}

For $1\leq i_0 \leq s$, define $\beps_{H(\fn),i_0} = (\eps_1,\dots ,\eps_s)\in \E$ by  
$$\eps_i := \begin{cases}
(-1)^{\deg \fp_i} & \text{ if $i \neq i_0$,}\\
-(-1)^{\deg \fp_i} & \text{ if $i = i_0$.}
\end{cases}$$
Since $s>1$ and there are at least two prime factors of $\fn$ with odd degree, 
$\beps_{H(\fn),i_0} \neq \mathbf{1}$, $\beps_{H(\fn),i_0} \neq \beps_{H(\fn)}$ for $1\leq i_0\leq s$.
Put
$$\tilde{f}_{i_0} := f - \frac{q f^*(1)}{1-q} E^{\beps_{H(\fn),i_0}}.$$
Then $\tilde{f}_{i_0}^*(\fm) = 0$ unless $\fp_{i_0} \mid \fm$. Hence by Lemma~\ref{lemLevelLow} there exists $g \in \cH(\fn/\fp_{i_0},R)$ such that
$\tilde{f}_{i_0} = g|B_{\fp_{i_0}}$ and 
$$\tilde{f}_{i_0}|W_{\fp_{i_0}} = (-1)^{\deg \fp_{i_0}} \left(f + \frac{q f^*(1)}{1-q} E^{\beps_{H(\fn),i_0}}\right) = g.$$ 
In particular, we get 
$$
(-1)^{ \deg \fp_{i_0}}g+g|B_{\fp_{i_0}} = 2 f$$ and 
$$(-1)^{ \deg \fp_{i_0}}g- g|B_{\fp_{i_0}} = (-1)^{ \deg \fp_{i_0}}g- g|W_{\fp_{i_0}} =  \frac{ 2 q f^*(1) }{1-q} \cdot E^{\beps_{H(\fp),i_0}}.$$
Adding this equations, we get $g=(-1)^{ \deg \fp_{i_0}}\left(f+\frac{q f^*(1) }{1-q} \cdot E^{\beps_{H(\fp),i_0}}\right)$. Since $f$ is cuspidal,
\begin{align*}
g^0(1) &= (-1)^{ \deg \fp_{i_0}} \frac{q f^*(1)}{1-q} \cdot (E^{\beps_{H(\fn),i_0}})^0(1) \\
&= \frac{ q^2f^*(1) }{q-1} \cdot N(\fn/\fp_{i_0},\beps_{H(\fn/\fp_{i_0})}) \cdot \frac{|\fp_{i_0}|-(-1)^{ \deg \fp_{i_0}}}{q+1}. 
\end{align*}
Therefore, using Lemma \ref{lem_new14}, we obtain 
\begin{align}\label{eqn 3.2}
0 &= (-1)^{ \deg \fp_{i_0}}g^0(1)+ (g|B_{\fp_{i_0}})^0(1) \\ 
\nonumber & = \frac{q^2}{1-q} \cdot \left(N(\fn,\beps_{H(\fn)}) \cdot \frac{1-(-1)^{ \deg \fp_{i_0}}|\fp_{i_0}|}{q+1}\right) \cdot f^*(1).
\end{align}

By assumption
$$\left(q+1, \frac{1-(-1)^{ \deg \fp_{1}}|\fp_{1}|}{q+1},\dots ,\frac{1-(-1)^{ \deg \fp_{s}}|\fp_{s}|}{q+1}\right) 
= (q+1,\deg \fp_1,...,\deg \fp_s) \in R^{\times},$$
so the equations~(\ref{eqn 3.1}) and (\ref{eqn 3.2}) force $f^*(1) \in R[N(\fn,\beps_{H(\fn)})]$.
From the Eisenstein property of $f$, we get $N(\fn,\beps_{H(\fn)})f^*(\fm) = 0$ for every $\fm$ coprime to $\fn$.
Therefore, by Proposition~\ref{prop-cor2.4}, we obtain $N(\fn,\beps_{H(\fn)}) f \in \cH_0(1,R) = 0$.
Finally, from the fact that $(q+1,N(\fn,\beps_{H(\fn)}))$ is a power of $2$, thus invertible in $R$, we conclude that 
there exists $\alpha \in R$ such that
$$f = \alpha (q+1)^2 f = \frac{q(q+1)\alpha f^*(1)}{1-q} E^{\beps_{H(\fn)}}.$$
Thus, $f \in R[N(\fn,\beps_{H(\fn)})] \cdot E^{\beps_{H(\fn)}}$.
\end{proof} 

Given a prime $\ell$ and a positive integer $r$, let $R_{\ell}^{(r)} := \Z_\ell[\zeta_p]/ \ell^r \Z_\ell[\zeta_p]$.
For our purposes, it suffices to focus on this particular coefficient ring for the remaining case:

\begin{prop}\label{prop 3.9}
Let $\ell$ be a prime such that $\ell \nmid q(q-1)$, $\ell \mid q+1$, and $\ell \mid \deg \fp_i$ for $1\leq i \leq s$. Then 
$$\widetilde{\cE}(\fn,R_{\ell}^{(r)})^{\beps_{H(\fn)}} = R_\ell^{(r)} \cdot E^{\beps_{H(\fn)}}, \quad \forall r \geq 1,$$
and 
$${\cE}_0(\fn,R_\ell^{(r)})^{\beps_{H(\fn)}} = R_\ell^{(r)}\left[N(\fn,\beps_{H(\fn)}\right] \cdot E^{\beps_{H(\fn)}} = 0.
$$
\end{prop}

\begin{proof} 
Note that the statement about ${\cE}_0(\fn,R_\ell^{(r)})^{\beps_{H(\fn)}}$ follows from the first statement 
by an argument similar to the argument in the proof of Proposition \ref{prop 3.2FT} combined with Remark \ref{remNnu}. 
Thus, it suffices to prove the first statement. Let $R = R_\ell^{(r)}$, which says that $R = R[\ell^r]$.\\

First, assume $r=1$. By Proposition~\ref{prop 3.4} the case when $s=1$ holds. Suppose $s>1$. Let $\beps_{H(\fn),s}$ be 
defined as in Lemma \ref{lem 3.8}. 
In particular, $\beps_{H(\fn),s} \neq \mathbf{1}, \beps_{H(\fn)}$.
Given $f \in \widetilde{\cE}(\fn,R)^{\beps_{H(\fn)}}$, put
$$\tilde{f} := f - \frac{qf^*(1)}{1-q} E^{\beps_{H(\fn),s}}.$$
By Proposition~\ref{propALOhta}, $\tilde{f}^*(\fm) = 0$ unless $\fp_{s} \mid \fm$; 
Lemma~\ref{lemLevelLow} then says that there exists $g \in \widetilde{\cH}(\fn/\fp_{s},R)$ such that
$\tilde{f} = g|B_{\fp_{s}}$ and 
$$\tilde{f}|W_{\fp_{s}} = (-1)^{\deg \fp_s} \left(f + \frac{q f^*(1)}{1-q} E^{\beps_{H(\fn),s}}\right) = g.$$ 
In particular, we have 
\begin{itemize}
\item[(i)] $$g+(-1)^{\deg \fp_s} g|B_{\fp_{s}} = 2(-1)^{\deg \fp_s}  f,$$
\item[(ii)] 
$$g- (-1)^{\deg \fp_s} g|B_{\fp_{s}} = g- (-1)^{\deg \fp_s} g|W_{\fp_{s}} = \frac{2(-1)^{\deg \fp_s} q f^*(1)}{1-q} \cdot E^{\beps_{H(\fp),s}}.$$
\end{itemize}

Consider the trace map : $\text{Tr}^{\fn}_{\fn/\fp_s}: \widetilde{\cH}(\fn,R) \rightarrow \widetilde{\cH}(\fn/\fp_s,R)$ defined by:
$$\text{Tr}^{\fn}_{\fn/\fp_s}(h):= h + h | W_{\fp_s} U_{\fp_s}.$$
Since $\ell \cdot f^*(1)= 0$, by Lemma~\ref{lem 3.8} we get $\text{Tr}^{\fn}_{\fn/\fp_s}(f^*(1) \cdot E^{\beps_{H(\fn),s}} ) = 0$ and
$$\text{Tr}^{\fn}_{\fn/\fp_s}(g) = (-1)^{\deg \fp_s} \text{Tr}^{\fn}_{\fn/\fp_s}(g |W_{\fp_s}).$$
The left hand side is equal to $(|\fp_s|+1)g$ as $g$ is of level $\fn/\fp_s$, and the right hand side is nothing but $(-1)^{\deg \fp_s} g\big| T_{\fp_s}$.
This implies that $g \in \widetilde{\cE}(\fn/\fp_s,R)^{\beps_{H(\fn/\fp_s)}}$ (when $\deg \fp_s$ is odd, we have $(|\fp_s|+1)g = -(|\fp_s|+1)g = 0$).
By induction on $s$, there exists $a \in R$ so that $g = a E^{\beps_{H(\fn/\fp_s)}}$,
and hence
$$f = 2^{-1}\left(g+ g|B_{\fp_s}\right) = 2^{-1} a \cdot E^{\beps_{H(\fn)}}.$$
${}$\\

Now assume $r>1$.
Given $f \in \widetilde{\cE}(\fn,R)^{\beps_{H(\fn)}}$, consider $f \bmod \ell \in \widetilde{\cE}(\fn,R/\ell R)^{\beps_{H(\fn)}}$. 
Since this latter module is spanned by $E^{\beps_{H(\fn)}}$, there exists $a \in R$ such that 
$$ f \equiv a \cdot E^{\beps_{H(\fn)}} \bmod \ell,$$
or equivalently 
$$f - a E^{\beps_{H(\fn)}} \in \widetilde{\cE}(\fn,\ell R)^{\beps_{H(\fn)}}.$$
Note that $\ell R$ is the maximal ideal in $R$, and the isomorphism $\iota_\ell : R/\ell^{r-1}R \cong \ell R$ (multiplication by $\ell$) induces the following (group) isomorphism:
$$\widetilde{\cE}(\fn,R/\ell^{r-1}R)^{\beps_{H(\fn)}} \cong \widetilde{\cE}(\fn,\ell R)^{\beps_{H(\fn)}}.$$
By the induction on $r$, there exists $b \in R/\ell^{r-1}R$ such that 
$$ f - a E^{\beps_{H(\fn)}} = \iota_\ell(b) E^{\beps_{H(\fn)}} \in \widetilde{\cE}(\fn,\ell R)^{\beps_{H(\fn)}}.$$
This completes the proof.
\end{proof}

The results of this section imply the following: 

\begin{thm}\label{thm 3.6FT} Let $\fn = \prod_{i=1}^s \fp_i \lhd A$ be a square-free ideal. 
Given a prime number $\ell$ not dividing $q(q-1)$ and a positive integer $r$,
we have that for $\beps \in \E$,
$$\cE_0(\fn,\Z/\ell^r\Z)^\beps \cong \begin{cases}
\Z/(\ell^r,\frac{N(\fn,\beps)}{\nu(\beps)})\Z, & \text{ if $\beps \neq \mathbf{1}$};\\
0, & \text{ if $\beps = \mathbf{1}$.} 
\end{cases}$$
\end{thm}

\begin{proof}\label{rem3.13} 
This follows from the observation that the coefficient ring $\Z_\ell[\zeta_p]/\ell^r\Z_\ell[\zeta_p]$ is a free $\Z /\ell^r\Z$-module; cf.\ 
the proof of Corollary 6.9 in \cite{Pal}. 
\end{proof}


\section{Proof of the main theorem}\label{sCDG}

Let $\Omega$, $X_0(\fn)$, $J_0(\fn)$, $\cC(\fn)$, $\cT(\fn)$ be as in the introduction. 
The Hecke operators $T_\fp$ may also be defined as a correspondences on $X_0(\fn)$, so 
$T_\fp$ induce endomorphisms of $J_0(\fn)$.  Similarly, the Atkin-Lehner 
involutions induce automorphisms of $J_0(\fn)$.  
Let $\cO(\Omega)^\times$ be the group of non-vanishing holomorphic rigid-analytic functions on $\Omega$. 
The group $\GL_2(\Fi)$ act on $\cO(\Omega)^\times$ via $(f|\gamma)(z)=f(\gamma z)$. 
To each $f\in \cO(\Omega)^\times$ van der Put associated a 
harmonic cochain $r(f)\in \cH(\sT, \Z)$ so that the sequence  
\begin{equation}\label{eqvdPut}
0\to \C_\infty^\times \to \cO(\Omega)^\times\xrightarrow{r} \cH(\sT, \Z)\to 0
\end{equation}
is exact and $\GL_2(\Fi)$-equivariant; cf. \cite[Thm. 2.1]{vdPut}, \cite[(1.7.2)]{GR}.  
Let $\Delta(z)\in \cO(\Omega)^\times$ be the Drinfeld discriminant function defined on page 183 of \cite[p. 183]{Discriminant}. 
For a non-zero ideal $\fm\lhd A$, let 
$$\Delta_\fm(z):=\Delta|B_\fm(z)=\Delta(\fm z).$$

From now on we assume $\fn=\fp_1\cdots \fp_s$ is square-free with the given prime decomposition. Given $\beps\in \E$, we set 
$$
\Delta^\beps:=\prod_{\fd|\fn}\Delta_\fd^{\epsilon_\fd} \in \cO(\Omega)^\times. 
$$
(Recall that $\epsilon_\fd=\prod_{\fp_i|\fd}\eps_i$; see Notation \ref{notn2.3}.) 

The harmonic cochain $r(\Delta)$ has been extensively studied by Gekeler in \cite{Improper}, \cite{Discriminant}, where 
he shows that $r(\Delta)=q(1-q)H$; see Remark \ref{remH}. Therefore, by the same remark, for $e\in E(\sT)^+$,
$$
r(\Delta)(e)=(1-q)\tE(e) \quad \text{and}\quad r(\Delta)(\bar{e})=(q^2-1)+ (1-q)\tE(\bar{e}). 
$$
More importantly for us, we also get 
$$
E^\beps(e)=\frac{1}{\nu(\beps)(1-q)} r(\Delta^\beps)(e)
$$
for all $e\in E(\sT)^+$. The above equality holds for all $e\in E(\sT)$ if and only if $\beps\neq \bone$, since  
$E^\beps$ is harmonic if and only if $\beps\neq \bone$; see Lemma \ref{lemEeps}. 

\begin{lem}\label{lemrootDelta}
Let $\ell$ be a prime number not dividing $q(q-1)$. Let $\ell^n$ be the largest power of $\ell$ 
such that there exists an $\ell^n$-th root of $\Delta^\beps$ in $\cO(\Omega)^\times$. Then $\ell^n$ 
divides $\nu(\beps)$.   
\end{lem}
\begin{proof}The following argument is essentially due to Gekeler; cf. \cite[Cor. 3.5]{Discriminant}. Let $f\in \cO(\Omega)^\times$
be such that $f^{\ell^n}=\Delta^\beps$. Since $r(f)\in \cH(\sT, \Z)$ is $\G_\infty$-invariant, it has Fourier expansion. 
Moreover, the Fourier coefficients are in $\Z[p^{-1}]$. On the other hand, 
$$
r(f)^0(1)=\ell^{-n}r(\Delta^\beps)^0(1)=\frac{\nu(\beps)(1-q)}{\ell^n}(E^\beps)^0(1) =\frac{q N(\fn, \beps)(1-q)}{\ell^n},
$$ 
$$
r(f)^\ast(1)=\ell^{-n}r(\Delta^\beps)^\ast(1)=\frac{\nu(\beps)(1-q)}{\ell^n}(E^\beps)^\ast(1) = \frac{(1-q)^2(1+q)}{q\ell^n}. 
$$
Thus, $\ell^n$ divides $N(\fn, \beps)$ and $(q+1)$; thus $\ell^n$ divides $\nu(\beps)$; cf. Remark \ref{remNnu}. 
\end{proof}

The cusps of $X_0(\fn)$ are in natural bijection with 
the cusps of $\G_0(\fn)\bs \sT$, and this bijection is compatible with the action of Hecke operators and Atkin-Lehner involutions; 
cf. \cite[(2.6)]{GR}. We will use the notation introduced in Lemma \ref{lemCusps} for the cusps of $X_0(\fn)$. 
Define a cuspidal divisor $D^{\beps}$ on $X_0(\fn)$ by 
$$
D^\beps=\sum_{\fd|\fn}\epsilon_\fd [\fd].  
$$
Note that 
$$
\deg(D^\beps)=\sum_{\fd|\fn}\epsilon_\fd = \prod_{i=1}^n(1+\eps_i).
$$
Hence $\deg(D^\beps)=0$ if $\beps\neq \bone$, 
and we can consider the class of $D^\beps$ in $\cC(\fn)$, which by abuse of notation we denote by the same symbol. 
Let $\langle D^\beps \rangle$ be the finite cyclic subgroup generated by $D^\beps$ in $\cC(\fn)$.  

\begin{prop}\label{thmCDG}
Assume $\beps=(\eps_1, \dots, \eps_s)\neq \bone$. 
\begin{enumerate}
\item We have $W_{\fp_i}(D^\beps)=\eps_i D^\beps$, $1\leq i\leq s$. 
\item Let $\ell$ be a prime number not dividing $q(q-1)$. Let $N$ be the order of $D^\beps$ in $\cC(\fn)$. 
Then $\ord_\ell(N)\geq \ord_\ell(N(\fn, \beps)/\nu(\beps))$. 
\end{enumerate}
\end{prop}
\begin{proof} Using Lemma \ref{lemCusps}, we compute 
$$
W_{\fp_i}(D^\beps)=W_{\fp_i}\sum_{\fd|\fn,\  \fp_i| \fd} \epsilon_\fd [\fd] + 
W_{\fp_i}\sum_{\fd|\fn,\ \fp_i\nmid \fd} \epsilon_\fd [\fd] 
= \sum_{\fd|\fn,\ \fp_i| \fd} \epsilon_\fd [\fd/\fp_i] + \sum_{\fd|\fn,\ \fp_i\nmid \fd} \epsilon_\fd [\fd \fp_i] 
$$ 
$$
= \eps_i\sum_{\fd|\fn,\ \fp_i| \fd} \epsilon_{\fd/\fp_i} [\fd/\fp_i] + \eps_i\sum_{\fd|\fn,\ \fp_i\nmid \fd} \epsilon_{\fd \fp_i} [\fd \fp_i] 
=\eps_i D^\beps, 
$$
which proves part (1). 

By formulas (3.10) and (3.11) in \cite{Discriminant}, for divisors $\fm$ and $\fd$ of $\fn$ we have 
$$
\ord_{[\fm]}\Delta_\fd = \ord_{W_\fd [\fm]}\Delta = \frac{|\fn|\cdot |(\fm, \fd)^2|}{|\fm|\cdot |\fd|}. 
$$
Therefore
$$
\ord_{[\fm]}\Delta^\beps = \frac{|\fn|}{|\fm|}\sum_{\fd|\fn} \epsilon_\fd \frac{(\fm, \fd)^2}{|\fd|} 
= \frac{|\fn|}{|\fm|}\prod_{\fp_i|\fm}(1+\eps_i|\fp_i|)\prod_{\fp_i\nmid\fm}(1+\eps_i|\fp_i|^{-1})
$$
$$
= \prod_{\fp_i|\fm}(1+\eps_i|\fp_i|)\prod_{\fp_i\nmid\fm}(|\fp_i|+\eps_i)=\epsilon_{\fn/\fm}N(\fn, \beps)=\epsilon_\fm\epsilon_\fn N(\fn, \beps). 
$$
This implies 
$$
\mathrm{div}(\Delta^\beps)=\sum_{\fm|\fn}(\ord_{[\fm]}\Delta^\beps)[\fm]= \epsilon_\fn N(\fn, \beps) \sum_{\fm|\fn}\epsilon_\fm[\fm] 
= \epsilon_\fn N(\fn, \beps) D^\beps. 
$$
It is easy to see from this equality that if $\ord_\ell(N)< \ord_\ell(N(\fn, \beps)/\nu(\beps))$, then $\Delta^\beps$ 
has an $\ell^n$-th root in $\cO(\Omega)^\times$ with $n>\ord_\ell(\nu(\beps))$. But this would contradict 
Lemma \ref{lemrootDelta}.  
\end{proof}

\begin{thm}\label{thmLast}
Let $\ell$ be a prime number not dividing $q(q-1)$. 
Then 
$$
\cC(\fn)_\ell^\beps=\cT(\fn)_\ell^\beps\cong 
\begin{cases}
\Z_\ell\big/\frac{N(\fn, \beps)}{\nu(\beps)}\Z_\ell, & \text{if $\beps\neq \bone$};\\
0, & \text{if $\beps= \bone$}.
\end{cases}
$$
Moreover, if $\beps\neq \bone$, then $D^\beps$ generates $\cT(\fn)_\ell^\beps$. 
\end{thm}
\begin{proof} As is explained in \cite[$\S$7.1]{PW1}, 
the canonical specialization of $\cT(\fn)$ into the component group $\Phi_\infty$ at $\infty$ of the N\'eron model of $J_0(\fn)$ 
induces an injective homomorphism 
$$
\cT(\fn)_\ell\hookrightarrow \cE_{00}(\fn, \Z_\ell/\ell^n\Z_\ell)
$$
for any $n\geq \ord_\ell(\# \Phi_\infty)$. Moreover, as follows from the discussion in \cite[$\S$2.6]{PW2},  
the above homomorphism is compatible with the action of $\W$. Therefore, we have 
\begin{equation}\label{eq-inclusions}
\langle D^\beps \rangle_\ell \subseteq \cC(\fn)_\ell^\beps\subseteq \cT(\fn)_\ell^\beps \subseteq 
\cE_{00}(\fn, \Z_\ell/\ell^n\Z_\ell)^\beps\subseteq \cE_{0}(\fn, \Z_\ell/\ell^n\Z_\ell)^\beps. 
\end{equation}

By Theorem \ref{thm 3.6FT}, if $\beps=\bone$ then $\cE_{0}(\fn, \Z_\ell/\ell^n\Z_\ell)^\beps=0$, and the claim of the theorem follows. 
Next, assume $\beps\neq \bone$. By Theorem \ref{thm 3.6FT}, 
$\cE_{0}(\fn, \Z_\ell/\ell^n\Z_\ell)^\beps \subseteq \Z_\ell\big/\frac{N(\fn, \beps)}{\nu(\beps)}\Z_\ell$. 
On the other hand, by Proposition \ref{thmCDG}, we have $\Z_\ell\big/\frac{N(\fn, \beps)}{\nu(\beps)}\Z_\ell\subseteq \langle D^\beps\rangle_\ell$. 
Comparing the orders of the groups, one immediately deduces that the inclusions in (\ref{eq-inclusions}) 
are equalities, which implies the claim of the theorem in this case.  
\end{proof}

\begin{rem} 
The Jacobian variety $J_0(\fn)$ has bad reduction at $\infty$ and at the primes dividing $\fn$. 
A crucial role in the previous proof plays the fact that the canonical specialization $\wp_\infty: \cT(\fn)_\ell\to \Phi_\infty$ 
is injective if $\ell\nmid (q-1)$. One can also consider the specializations $\wp_{\fp_i}: \cT(\fn)\to \Phi_{\fp_i}$ 
into the component groups of $J_0(\fn)$ at finite places. The behaviour of these maps is somewhat different from $\wp_\infty$. 
Without loss of generality, we consider $\wp_{\fp_1}$. 
The reduction $X_0(\fn)_{\F_{\fp_1}}$ consists of two irreducible components $Z$ and $Z'$, both isomorphic to $X_0(\fn/\fp_1)_{\F_{\fp_1}}$, 
intersecting transversally at certain points. The Atkin-Lehner involution $W_{\fp_1}$ interchanges $Z$ and $Z'$.  
The reductions of the cusps lie in the smooth locus, and no two cusps reduce to the same 
point. Moreover, the reductions of all $[\fm]$, with $\fp_1\nmid \fm$, lie on one component, say $Z$, and the reductions 
of $[\fm]$, with $\fp_1|\fm$, lie on the other component $Z'$. This implies that 
$$
\wp(D^\beps) = \prod_{i=2}^s(1+\eps_i)Z+\eps_1\prod_{i=2}^s(1+\eps_i)Z'=(1-\eps_1)\prod_{i=2}^s(1+\eps_i)(Z-Z'). 
$$
Let $z:=Z-Z'$. We conclude that $\wp(D^\beps)=2^s z$ if $\beps=(-1, 1,\dots, 1)$, and $\wp(D^\beps)=0$ 
otherwise. By the argument in the proof of Lemma 8.3 in \cite{PW2}, if $\ell\nmid (q+1)$, then 
$$
(\Phi_{\fp_1})_\ell = \langle z\rangle_\ell. 
$$
Moreover, by \cite[Thm. 5.3]{PW1}, if $\ell\nmid (q^2-1)$, then 
$$
(\Phi_{\fp_1})_\ell \cong \Z_\ell/N(\fn, \beps)\Z_\ell, \text{ where $\beps=(-1, 1,\dots, 1)$}. 
$$ 
Overall, we see that for $\ell\nmid (q^2-1)$, the map $\wp_{\fp_1}:\cC(\fn)_\ell^\beps \to (\Phi_{\fp_1})_\ell$ 
is an isomorphism if $\beps =(-1, 1,\dots, 1)$, and is $0$, otherwise. 
\end{rem}

\begin{prop}\label{propCp=0} Assume $\fn=\fp_1\cdots \fp_s$ is square-free. 
Then the exponent of $\cC(\fn)$ divides 
$$
\rho(\fn):= \prod_{i=1}^s\left((|\fp_i|-1)(|\fp_i|+1)^{s-1}\right).$$
In particular, $\cC(\fn)_p=0$. 
\end{prop}

\begin{proof} When $\fn$ is prime (i.e., $s=1$), this follows from a result of Gekeler; cf. \cite[Cor. 3.23]{Discriminant}. 
Now suppose $s>1$. 
Take $\fp = \fp_1$ and $\fn':= \fn/\fp$. 
Let $\pi_\fp: X_0(\fn) \rightarrow X_0(\fn')$ be the canonical projection map, 
and $\widetilde{\pi}_\fp:= \pi_\fp \circ W_\fp$. For each cusp $[\fm]_{\fn'}$ of $X_0(\fn')$ ($\fm \mid \fn'$), we have
\begin{align*}
\pi_\fp^*([\fm]_{\fn'}) &= |\fp|[\fm]_{\fn} + [\fm\fp]_{\fn} \in \Div(X_0(\fn)) \\ 
\widetilde{\pi}_\fp^*([\fm]_{\fn'}) &= |\fp|[\fm\fp]_{\fn} + [\fm]_{\fn} \in \Div(X_0(\fn)).
\end{align*}
Thus 
\begin{align*}
|\fp| \cdot \pi_\fp^*([\fm_1]_{\fn'}-[\fm_2]_{\fn'}) - \widetilde{\pi}_\fp^*([\fm_1]_{\fn'}-[\fm_2]_{\fn'})
&= (|\fp|^2-1)([\fm_1]_\fn - [\fm_2]_\fn),\\
|\fp| \cdot \widetilde{\pi}_\fp^*([\fm_1]_{\fn'}-[\fm_2]_{\fn'}) - \pi_\fp^*([\fm_1]_{\fn'}-[\fm_2]_{\fn'})
&= (|\fp|^2-1)([\fm_1\fp]_\fn - [\fm_2\fp]_\fn),
\end{align*}
and
$$\pi_\fp^*([\fm_1]_{\fn'}-[\fm_2]_{\fn'}) - \widetilde{\pi}_\fp^*([\fm_1]_{\fn'}-[\fm_2]_{\fn'})
= (|\fp|-1)\big(([\fm_1]_\fn - [\fm_1 \fp]_\fn)-([\fm_2]_\fn-[\fm_2\fp]_\fn)\big).$$
By induction hypothesis we have $\rho(\fn')\cdot \big([\fm_1]_{\fn'}-[\fm_2]_{\fn'}\big) = 0 \in \cC(\fn')$ for $\fm_1,\fm_2 \mid \fn'$.
Therefore
$$\rho(\fn')(|\fp|^2-1)([\fm_1]_\fn-[\fm_2]_\fn) = \rho(\fn')(|\fp|^2-1)([\fm_1\fp]_\fn-[\fm_2\fp]_\fn) = 0 \in \cC(\fn)$$
and
$$\rho(\fn')(|\fp|-1)\Big(([\fm_1]_\fn - [\fm_1 \fp]_\fn)-([\fm_2]_\fn-[\fm_2\fp]_\fn)\Big) = 0 \in \cC(\fn)$$
for every $\fm_1,\fm_2 \mid \fn'$.
Moreover, we have
$$\text{div}\left(\frac{\Delta}{\Delta_\fp}\right) = (|\fp|-1)\sum_{\fm \mid \fn'} \frac{|\fn'|}{|\fm|}([\fm]_\fn - [\fm\fp]_\fn) \in \Div^0(X_0(\fn)).$$
Thus for every $\fm_1 \mid \fn'$,
\begin{eqnarray}
\rho(\fn') \cdot \text{div}\left(\frac{\Delta}{\Delta_\fp}\right)
&=& \rho(\fn')(|\fp|-1)\left(\sum_{\fm \mid \fn'} \frac{|\fn'|}{|\fm|}\right) \cdot ([\fm_1]_\fn-[\fm_1\fp]_\fn) \nonumber\\
&=& \rho(\fn')(|\fp|-1)\left(\prod_{\fp' \mid \fn'}(|\fp'|+1)\right) \cdot ([\fm_1]_\fn-[\fm_1\fp]_\fn). \nonumber
\end{eqnarray}
Hence we conclude that for every $\fm_1,\fm_2$ dividing $\fn$,
$$\left(\rho(\fn')\cdot(|\fp|^2-1)\prod_{\fp' \mid \fn'}(|\fp'|+1)\right) \cdot ([\fm_1]_\fn - [\fm_2]_\fn) = \rho(\fn)\cdot ([\fm_1]_\fn - [\fm_2]_\fn) = 0 \in \cC(\fn).$$
\end{proof}

\renewcommand{\bibliofont}{\normalsize}
\bibliographystyle{amsplain}
\bibliography{EisTorsion2.bib}

\end{document}